\documentclass[12pt]{amsart}
\usepackage{graphicx}
\usepackage{amsmath}
\usepackage{amsfonts}
\usepackage{amssymb}
\usepackage{color}

\newtheorem{thm}{Theorem}[section]

\newtheorem{lem}[thm]{Lemma}
\newtheorem{prop}[thm]{Proposition}
\newtheorem{defn}[thm]{Definition}

\newtheorem{rem}[thm]{Remark}

\newcommand{\ip}[2]{{\langle #1,#2\rangle}}
\newcommand{\1}{\mathbf 1}
\newcommand{\Var}{{\operatorname{Var}}}
\renewcommand{\Im}{{\mathcal I}}

\newcommand{\rr}{\mathbb R}
\newcommand{\E}{\mathbb E}
\newcommand{\R}{\mathbb R}

\begin{document}
\sloppy
\title[Invariance principle for tempered fractional time series models]{Invariance principle for tempered fractional time series models}


\author{Farzad Sabzikar}
\address{Farzad Sabzikar, Department of Statistics and Probability,
Michigan State University, East Lansing MI 48823}
\email{sabzika2@stt.msu.edu}

\begin{abstract}
Autoregressive tempered fractionally integrated moving average (ARTFIMA) time series is a useful model for velocity data in turbulence flows. In this paper, we obtain
an invariance principle for the partial sum of an ARTFIMA process. The limiting process is called tempered Hermite process of order one, $THP^{1}$, which is well-defined for any $H>\frac{1}{2}$. When $\frac{1}{2}<H<1$, we develop the Wiener integral with respect to $THP^{1}$ to provide the sufficient condition for the convergence
\begin{equation*}
n^{-H}\sum_{k=0}^{+\infty}f\Big(\frac{k}{n}\Big)X^{\frac{\lambda}{n}}_{k}\rightarrow \int_{\rr}f(u)Z^{1}_{H,\lambda}(du)
\end{equation*}
in distribution, as $n\to\infty$, where $X_{k}$ is an ARTFIMA time series and $Z^{1}_{H,\lambda}$ is $THP^{1}$.
\end{abstract}

\maketitle

\section{Introduction}
The motivation of this work comes from the application of stochastic processes in the theory of turbulence. Kolmogorov \cite{KolmogorovFBM, Friedlander} proposed a model for the energy spectrum of turbulence in the inertial range, predicting that the spectrum $f(k)$ would follow a power law $f(k)\propto k^{-5/3}$ where $k$ is the frequency.

Figure 1 
illustrates the complete Kolmogorov spectral model for turbulence, and the power law approximation in the inertial range.  Large eddies are produced in the low frequency range. In the inertial range, larger eddies are continuously broken down into smaller eddies, until they eventually dissipate, in the high frequency range.

\begin{figure}[ht]\label{fig4}
\vskip-40pt
\includegraphics[scale=0.5]{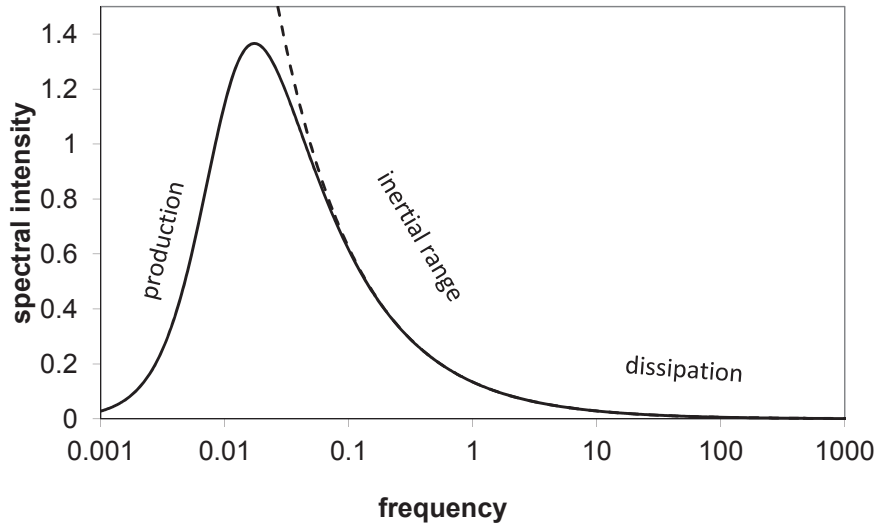}
\vskip-40pt
\caption{Kolmogorov spectral density (solid line) and power law approximation in the inertial range (dotted line), from \cite{Meerschaertsabzikarkumarzeleki}.}
\end{figure}

The autoregressive tempered fractionally integrated moving average (ARTFIMA) time series modifies the coefficient of an autoregressive tempered fractionally integrated moving average (ARFIMA)
model by multiplying an exponential tempering factor.
The spectral density of the ARTFIMA $(0,\alpha,\lambda,0)$ is proportional to $\left|e^{-(\lambda+ik)}-1\right|^{-2\alpha}\approx (\lambda^2+k^2)^{-\alpha}$ when $k,\lambda$ are sufficiently small. For small values of the tempering parameter $\lambda$, the spectral density of an ARTFIMA $(0,\alpha,\lambda,0)$ time series grows like $k^{-2\alpha}$ as $|k|$ decreases, but remains bounded as $|k|\to 0$, in agreement with the general theory of turbulence illustrated in Figure 1. We refer the reader to \cite{Meerschaertsabzikarkumarzeleki} to see the application of the ARTFIMA time series for turbulence in geophysical flows.

As it was mentioned, the ARTFIMA time series can be a useful discrete time stochastic model for turbulence. In this paper, we are interested to answer several questions which are related with the ARTFIMA time series. The first question is:
\begin{description}
  \item[1] Assume $X_{k}$ follows an ARTFIMA time series model. When do we have an invariance principle
  \begin{equation*}
  n^{-H}\sum_{k=0}^{[nt]}X_{k}\Rightarrow Y(t)
  \end{equation*}
  and what is the limiting process $Y$?
\end{description}
We prove that $\{Y(t)\}_{t\geq 0}$ is a Gaussian process which interpolates between fractional Brownian motion (FBM) and the standard  Ornstein Uhlenbeck process with the time domain representation
\begin{equation*}
Y(t):=Z^{1}_{H,\lambda}(t)={\int_{\rr}\int_{0}^{t}\Big({(s-y)_{+}^{H-\frac{3}{2}}}e^{-\lambda(s-y)_{+}}\Big)\ ds\ B(dy)},
\end{equation*}
where $(x)_{+}=xI(x>0)$, $B(dy)$ is an independently scattered Gaussian random measure on $\rr$ with control measure $\sigma^2\ dx$, $H>\frac{1}{2}$ and $\lambda>0$.
The process $Z^{1}_{H,\lambda}$ is called tempered Hermite process of order one, $THP^{1}$.

We called $Z^{1}_{H,\lambda}$ as tempered Hermite process of order one since it can be extended to     \begin{equation*}\label{eq:THPdefnorderk}
Z_{H,\lambda}^{k}(t):={\int_{\rr^k}^{'}\int_{0}^{t}\Big(\prod^{k}_{i=1}{(s-y_i)_{+}^{-(\frac{1}{2}+\frac{1-H}{k})}}e^{-\lambda(s-y_i)_{+}}\Big)\ ds\ B(dy_1)\ldots B(dy_k)}
\end{equation*}
for any $k\geq 2$. The prim on the integral sign shows that one does not integrate on diagonals where $y_{i}=y_{j}$, $i\neq j$.  The Hermite process \cite{Taqqu,Dobrushinmajor} is a special case of $\{Z_{H,\lambda}^{k}\}$ with $\lambda=0$. Unlike the Hermite process, tempered Hermite process of order $k$ is well-defined also for any $H>\frac{1}{2}$ because the exponential tempering keeps the integrand in $L^{2}(\rr^k)$. 



The second question is naturally the extension of the first question:
\begin{description}
  \item[2] Suppose $f$ is a deterministic function. What is the sufficient condition for the convergence
\begin{equation*}
n^{-H}\sum_{k=0}^{+\infty}f\Big(\frac{k}{n}\Big)X_{k}\rightarrow \int_{\rr}f(u)Z^{1}_{H,\lambda}(du)
\end{equation*}
in distribution, as $n\to\infty$?
\end{description}

In order to provide the sufficient condition for the convergence (in distribution) in the second question, we first require to develop the Wiener integral $\int_{\rr}f(u)Z^{1}_{H,\lambda}(du)$, where $f$ is a deterministic functions in an appropriate space and $\frac{1}{2}<H<1$. Our approach to develop the Wiener integral with respect to $Z^{1}_{H,\lambda}$ is based on tempered fractional calculus. In fact, we show that a representation of $Z^{1}_{H,\lambda}$ based on fractional calculus. That is,
\begin{equation*}\label{eq:connection with TFI}
Z^{1}_{H,\lambda}(t)
=\Gamma(H-\frac{1}{2})\int_{-\infty}^{+\infty}\Big({\mathbb{I}}^{H-\frac{1}{2},\lambda}_{-}{\1}_{[0,t]}\Big)(x)\ B(dx),
\end{equation*}
where $H>\frac{1}{2}$ and $\Big({\mathbb{I}}^{\alpha,\lambda}_{-}f\Big)(x)$ is tempered fractional integral of order $\alpha$ of a function $f$. We refer the reader to \cite{MeerschaertsabzikarSPA} for more details on tempered fractional integrals and derivatives. This representation enables us to characterize the classes of deterministic functions $f$ for which the Wiener integral $\int_{\rr}f\ Z^{1}_{H,\lambda}(du)$ is well defined. On the other hand, we need to study the asymptotic behavior of the spectral density of the ARTFIMA $(0,\alpha,\lambda,0)$ for low frequency. Therefor, we can prove the convergence in distribution that we had in the second question.

The paper is organized as follows. In Section \ref{sec2}, we define tempered Hermite process of order one, $\{Z^{1}_{H,\lambda}\}$, using a time domain representation, and we develop the spectral domain representation of $\{Z^{1}_{H,\lambda}\}$. In Section \ref{sec3}, we answer the second question by investigating the Wiener integral with respect to tempered Hermite process of
order one. In Section \ref{sec4}, we recall the definition of the autoregressive tempered fractionally integrated moving average (ARTFIMA) time series and some of its basic properties such as the covariance function and spectral density. The answer of the first and third question are given in Section \ref{sec5}. Some definitions and lemmas which are related to fractional calculus are contained in Appendix.

\section{Time and spectral domain representations }\label{sec2}

In this section, we define tempered Hermite processes of order one, $THP^{1}$. We start with the time domain representation.

Let $\{B(t)\}_{t\in \rr}$ be a real-valued Brownian motion on the real line, a process with stationary independent increments such that $B(t)$ has a Gaussian distribution with mean zero and variance $|t|$ for all $t\in\rr$.  Define an independently scattered Gaussian random measure $B(dx)$ with control measure $m(dx)=dx$ by setting $B[a,b]=B(b)-B(a)$ for any real numbers $a<b$, and then extending to all Borel sets.  Since Brownian motion sample paths are almost surely of unbounded variation, the measure $B(dx)$ is not almost surely $\sigma$-additive, but it is a $\sigma$-additive measure in the sense of mean square convergence.  Then the stochastic integrals $I(f):=\int f(x)B(dx)$ are defined for all functions $f:\rr\to\rr$ such that $\int f(x)^2dx<\infty$, as Gaussian random variables with mean zero and covariance $\E[I(f)I(g)]=\int f(x)g(x)dx$.  See for example \cite[Chapter 3]{SamorodnitskyTaqqu} or \cite[Section 7.6]{FCbook}.

\begin{defn}\label{defTHP}
Given an independently scattered Gaussian random measure $B(dx)$ on $\rr$ with control measure $\sigma^2 dx$, $H>\frac{1}{2}$, $\lambda>0$, the stochastic integral
\begin{equation}\label{eq:THPdefn}
Z_{H,\lambda}^{1}(t):={\int_{\rr}\int_{0}^{t}\Big({(s-y)_{+}^{H-\frac{3}{2}}}e^{-\lambda(s-y)_{+}}\Big)\ ds\ B(dy)}
\end{equation}
where $(x)_{+}=xI(x>0)$, will be called a {\it tempered Hermit process} of order one \textrm{$(THP^1)$}.
\end{defn}
The next lemma shows that $Z_{H,\lambda}^{1}(t)$ is well-defined for any $t>0$.
\begin{lem}\label{lem:g squar integrable}
The function
\begin{equation}\label{eq:integrand}
g_{H,\lambda,t}(y):=\int_{0}^{t}(s-y)_{+}^{H-\frac{3}{2}}e^{-\lambda(s-y)_{+}}\ ds
\end{equation}
is square integrable over the entire real line for any $H>\frac{1}{2}$ and $\lambda>0$.
\end{lem}
\begin{proof}
The proof is similar to \cite[Theorem 3.5]{BaiTaqqu}. To show that $g_{H,\lambda,t}(y)$ is square integrable over the entire real line, we write
\begin{equation*}
\begin{split}
\int_{\rr}g_{H,\lambda,t}(y)^{2}\ dy&=\int_{\rr}\int_{0}^{t}\int_{0}^{t}(s_1-y)_{+}^{H-\frac{3}{2}}e^{-\lambda(s_1-y)_{+}}(s_2-y)_{+}^{H-\frac{3}{2}}e^{-\lambda(s_2-y)_{+}}ds_1\ ds_2
\ dy\\
&=2\int_{0}^{t}ds_1\int_{s_1}^{t}ds_2\int_{\rr}(s_1-y)_{+}^{H-\frac{3}{2}}e^{-\lambda(s_1-y)_{+}}(s_2-y)_{+}^{H-\frac{3}{2}}e^{-\lambda(s_2-y)_{+}}
\ dy\\
&=2\int_{0}^{t}ds\int_{0}^{t-s}du\int_{\rr}(w)_{+}^{H-\frac{3}{2}}e^{-\lambda(w)_{+}}(w+u)_{+}^{H-\frac{3}{2}}e^{-\lambda(w+u)_{+}}
\ dw\\
&\qquad (s=s_1, u=s_2-s_1, w=s_1-y)\\
&=2\int_{0}^{t}ds\int_{0}^{t-s}u^{2H-2}e^{-\lambda u}du\int_{0}^{+\infty}{x}^{H-\frac{3}{2}}(1+x)^{H-\frac{3}{2}}e^{-2\lambda ux}\ dx\\
&=\frac{2\Gamma(H-\frac{1}{2})}{\sqrt{\pi}(2\lambda)^{H-1}}\int_{0}^{t}ds\int_{0}^{t-s}u^{H-1}K_{1-H}(\lambda u)\ du,
\end{split}
\end{equation*}
where we applied a standard integral formula \cite[Page 344]{Gradshteyn}
\begin{equation}\label{eq:standard integral formula}
\int_{0}^{\infty}x^{\nu-1}(x+\beta)^{\nu-1}e^{-\mu x}\ dx=\frac{1}{\sqrt{\pi}}\big(\frac{\beta}{\mu}\big)^{\nu-\frac{1}{2}}e^{\frac{\beta\mu}{2}}
\Gamma(\nu)K_{\frac{1}{2}-\nu}(\frac{\beta\mu}{2}),
\end{equation}
for $|\arg \beta|<\pi$, $Re \mu>0$, $Re \nu>0$. Here $K_{\nu}(x)$ is modified Bessel function of the second kind (see Appendix for more details about $K_{\nu}(x)$). Next, we need to show that the last integrals is finite for any $H>\frac{1}{2}$. First, assume $\frac{1}{2}<H<1$. In that case, $K_{1-H}(\lambda u)\sim u^{H-1}$ as $u\to 0$ (\cite[Chapter 9]{abramowitz}), and hence the integrand $u^{H-1}K_{1-H}(\lambda u)\sim u^{2H-2}$ ,as $u\to 0$, which is integrable provided that $H>\frac{1}{2}$. Now, let $H>1$. In the later case, $K_{1-H}(\lambda u)\sim u^{1-H}$ as $u\to 0$ and therefore the integrands $u^{H-1}K_{1-H}(\lambda u)\sim C$, $C$ is a constant, which in integrable and this completes the proof.
\end{proof}
\begin{rem}
{\emph {When $\lambda=0$, the right-hand side of \eqref{eq:THPdefn} is a fractional Brownian motion (FBM), a self-similar Gaussian stochastic process with Hurst scaling index $H$ (e.g., see \cite{EmbrechtsMaejima}).
When $\lambda=0$ and $H>1$, the right-hand side of \eqref{eq:THPdefn} does not exist, since the integrand is not in $L^{2}(\mathbb{R})$.  However, $THP^1$ with $\lambda>0$ and $H>1$ is well-defined, because the exponential tempering keeps the integrand in $L^{2}(\mathbb{R})$.}}
\end{rem}
We now compute the covariance function $R(t,s)=\E[Z^{1}_{H,\lambda}(t)Z^{1}_{H,\lambda}(s)]$ of $THP^1$.
\begin{prop}\label{prop:THPcovariance}
The process $Z^{1}_{H,\lambda}$ given by \eqref{eq:THPdefn} has the covariance function
\begin{equation}\label{eq:THPacvf}
R(t,s)=\frac{2\Gamma(H-\frac{1}{2})}{\sqrt{\pi}(2\lambda)^{H-1}}
\int_{0}^{t}\int_{0}^{s}|u-v|^{H-1}K_{1-H}(\lambda|u-v|)dv\ du.
\end{equation}
\end{prop}

\begin{proof}
The proof is similar to that of Lemma \ref{lem:g squar integrable}. By applying Fubini and the It\^{o} isometry, we have
\begin{equation*}
\begin{split}
R(t,s)&=2\int_{\rr}\int_{0}^{t}\int_{0}^{s}(u-y)_{+}^{H-\frac{3}{2}}
(v-y)_{+}^{H-\frac{3}{2}}e^{-\lambda(u-y)_{+}}e^{-\lambda(v-y)_{+}}dv\ du\ dy\\
&=2\int_{0}^{t}\int_{0}^{s}\int_{-\infty}^{v}
(u-y)^{H-\frac{3}{2}}(v-y)^{H-\frac{3}{2}}e^{\lambda (u-y)}e^{-\lambda(v-y)} dy\ dv\ du\\
&\qquad (\textrm{assume}\quad v<u)\\
&=2\int_{0}^{t}\int_{0}^{s}(u-v)^{2H-2}e^{-\lambda (u-v)}\int_{0}^{+\infty}
{x}^{H-\frac{3}{2}}(1+x)^{H-\frac{3}{2}}e^{-2\lambda (u-v)x} dx\ dv\ du\\
&\qquad (u-y=x(u-v)+u-v)\\
&=\frac{2\Gamma(H-\frac{1}{2})}{\sqrt{\pi}(2\lambda)^{H-1}}
\int_{0}^{t}\int_{0}^{s}(u-v)^{H-1}K_{1-H}(\lambda(u-v))dv\ du,
\end{split}
\end{equation*}
where we applied \eqref{eq:standard integral formula} to get the last integral. Hence
\begin{equation*}
R(t,s)=\frac{2\Gamma(H-\frac{1}{2})}{\sqrt{\pi}(2\lambda)^{H-1}}
\int_{0}^{t}\int_{0}^{s}|u-v|^{H-1}K_{1-H}(\lambda|u-v|)dv\ du.
\end{equation*}
By using the same argument in Lemma \ref{lem:g squar integrable}, one can show that the covariance is finite for any $H>\frac{1}{2}$.
\end{proof}
The next results shows that $THP^1$ has a nice scaling property, involving both the time scale and the tempering. Here the symbol $\triangleq $ indicates the equivalence of finite dimensional distributions.

\begin{prop}\label{prop:sssi}
The process $Z^{1}_{H,\lambda}$ given by \eqref{defTHP} has stationary increments, such that
\begin{equation}\label{eq:scalingTHP}
\left\{Z_{H,\lambda}^{1}(ct)\right\}_{t\in\rr}{\triangleq}\left\{c^{H}Z_{H,c\lambda}^{1}(t)\right\}_{t\in\rr}
\end{equation}
for any scale factor $c>0$.
\end{prop}

\begin{proof}
Since $B(dy)$ has control measure $m(dy)=\sigma^2 dy$, the random measure $B(c\ dy)$ has control measure $c^{1/2} \sigma^2dy$. Given $t_j$, $j=1,\dots,n$, a change of variables ${s=cs'}$ and ${y=cy'}$  then yields
\begin{equation*}
\begin{split}
&\left(Z_{H,\lambda}^{1}(ct_j):j=1,\ldots,n\right)\\
&=\left({\int_{\rr}\int_{0}^{ct_j}\Big({(s-y)_{+}^{H-\frac{3}{2}}}e^{-\lambda(s-y)_{+}}\Big)ds\ B(dy)}\right)\\
&=\left({\int_{\rr}\int_{0}^{t_j}\Big({(cs'-cy)_{+}^{H-\frac{3}{2}}}e^{-\lambda(cs'-cy)_{+}}\Big)\ c\ ds'\ B(dy)}\right)\\
&\triangleq  c^{H}\left({\int_{\rr}\int_{0}^{t_j}\Big({(s'-y')_{+}^{H-\frac{3}{2}}}e^{-\lambda c(s'-y')_{+}}\Big)\ ds'\ B(c\ dy')}\right)\\
&=\left(c^{H}Z_{H,c\lambda}^{1}(t_j):j=1,\ldots,n\right)
\end{split}
\end{equation*}
so that \eqref{eq:scalingTHP} holds. Suppose now $s_j<t_j$, and change variables $x=x'+s$, $y=s+y'$ to get
\begin{equation*}
\begin{split}
&\left(Z_{H,\lambda}^{1}(t_j)-Z_{H,\lambda}^{1}(s_j):j=1,\ldots,n\right)\\
&=\left({\int_{\rr}\int_{s_j}^{t_j}\Big({(x-y)_{+}^{H-\frac{3}{2}}}e^{-\lambda(x-y)_{+}}\Big)dx B(dy)}\right)\\
&=\left({\int_{\rr}\int_{0}^{t_j-s_j}\Big({(x'+s-y)_{+}^{H-\frac{3}{2}}}e^{-\lambda(x'+s-y)_{+}}\Big)dx' B(dy)}\right)\\
&\triangleq \left({\int_{\rr}\int_{0}^{t_j-s_j}\Big({(x'-y')_{+}^{H-\frac{3}{2}}}e^{-\lambda(x'-y')_{+}}\Big)dx' B(dy')}\right)\\
&=\left(Z_{H,\lambda}^{1}(t_j-s_j):j=1,\ldots,n\right)
\end{split}
\end{equation*}
which shows that $THP^1$ has stationary increments.
\end{proof}

We next give another representation of $THP^1$ which is called the spectral domain representation. Let $\hat B_1$ and $\hat B_2$ be independent Gaussian random measures with $\hat B_{1}(A)=\hat B_{1}(-A)$, $\hat B_{2}(A)=-\hat B_{2}(-A)$ and $\E[(\hat B_{i}(A))^2]=m(A)/2$, where $m(dx)=\sigma^2 dx$, and define the  complex-valued Gaussian random measure $\hat B=\hat B_{1}+i\hat B_{2}$.  If $f(x)$ is a complex-valued function of $x$ real such that its Fourier transform $\hat f(\omega):=(2\pi)^{-1/2}\int e^{i\omega x}f(x)\,dx$ exists and $\int |\hat f(\omega)|^2 d\omega<\infty$, we define the stochastic integral $\hat I(\hat f)=\int \hat f(\omega)\hat B(d\omega):=\int \hat f_1(\omega) \hat B_1(d\omega)-\int \hat f_2(\omega)\hat B_2(d\omega)$, where $\hat f=\hat f_1+i\hat f_2$ is separated into real and imaginary parts.  Then $\hat I(\hat f)$ is a Gaussian random variable with mean zero, such that $\E[\hat I(\hat f)\hat I(\hat g)]=\int \hat f(\omega)\overline{\hat g(k)}\,d\omega$ for all such functions, and the Parseval identity $\int f(x)g(x)\,dx=\int \hat f(\omega) \overline{\hat g(\omega)}\,d\omega$ implies  that
$(\int f(x)B(dx),\int g(x)B(dx))\triangleq (\int \hat f(\omega)\hat B(d\omega),\int \hat g(\omega)\hat B(d\omega))$, see Proposition 7.2.7 in \cite{SamorodnitskyTaqqu}.

\begin{prop}\label{prop:THPdefharmo}
The process $Z^{1}_{H,\lambda}$ given by \eqref{defTHP} has the spectral domain representation
\begin{equation}\label{eq:THPdefharmo}
Z^{1}_{H,\lambda}(t)\triangleq \frac{1}{C(H)}\int_{\rr} \frac{e^{i\omega t}-1}{i\omega}
(\lambda+i\omega)^{\frac{1}{2}-H}
\widehat{B}(d\omega),
\end{equation}
where $C(H)=\frac{\sqrt{2\pi}}{\Gamma(H-\frac{1}{2})}$.
\end{prop}
\begin{proof}
To show that the stochastic integral \eqref{eq:THPdefharmo} exists, note that
$\Big|\frac{e^{i\omega t}-1}{i\omega}(\lambda+i\omega)^{\frac{1}{2}-H}\Big|^2$ is bounded for $\omega\rightarrow 0$ and behaves like $|\omega|^{-1-2H}$ , as $\omega\rightarrow \infty$,
which is integrable provided that $H>0$.
Observe that the function $g_{H,\lambda,t}$, given by \eqref{eq:integrand}, has the Fourier transform
\begin{equation*}
\begin{split}
\widehat{g_{H,\lambda,t}}(\omega)&=
\frac{1}{\sqrt{2\pi}}\int_{\rr}e^{i\omega y}
\int_{0}^{t}(s-y)^{H-\frac{3}{2}}_{+}e^{-\lambda(s-y)_{+}}\ ds\ dy\\
&=\frac{1}{\sqrt{2\pi}}\int_{\rr}e^{i\omega y}\int_{\rr}
(s-y)^{H-\frac{3}{2}}e^{-\lambda(s-y)}{\1}\{0<s<t\}{\1}\{s-y>0\}\ ds\ dy\\
&=\frac{1}{\sqrt{2\pi}}\int_{\rr}e^{i(s-u)\omega}\int_{\rr}
{u}^{H-\frac{3}{2}}e^{-\lambda u}{\1}\{0<s<t\}{\1}\{u>0\}\ ds\ du\\
&=\frac{1}{\sqrt{2\pi}}\int_{\rr}e^{is\omega}{\1}\{0<s<t\}\int_{\rr}
{u}^{H-\frac{3}{2}}e^{-(\lambda+i\omega)u}{\1}\{u>0\}\ du\ ds\\
&=\frac{\Gamma(H-\frac{1}{2})}{\sqrt{2\pi}}\frac{e^{i\omega t}-1}{i\omega}
(\lambda+i\omega)^{\frac{1}{2}-H}
\end{split}
\end{equation*}
provided that $ H>\frac{1}{2}$ and then by applying \eqref{eq:THPdefn}
\begin{equation*}
\begin{split}
Z^{1}_{H,\lambda}(t)&=\int_{-\infty}^{+\infty}g_{H,\lambda,t}(x)B(dx)\\
&\triangleq \int^{+\infty}_{-\infty}\widehat{g_{H,\lambda,t}}(\omega){\hat{B}}(d\omega)=
\frac{1}{C(H)}\int^{+\infty}_{-\infty}\frac{e^{i\omega t}-1}{i\omega}(\lambda+i\omega)^{\frac{1}{2}-H}{\hat{B}}(d\omega)
\end{split}
\end{equation*}
which is equivalent to \eqref{eq:THPdefharmo}.
\end{proof}

\begin{rem}
{\emph{We called the process $Z^{1}_{H,\lambda}$ as the tempered Hermite process of order one, since it is a special case of the following stochastic process which is called tempered Hermite process of order $k$:
\begin{equation}\label{eq:THPdefnorderk}
Z_{H,\lambda}^{k}(t):={\int_{\rr^k}\int_{0}^{t}\Big(\prod^{k}_{i=1}{(s-y_i)_{+}^{-(\frac{1}{2}+\frac{1-H}{k})}}e^{-\lambda(s-y_i)_{+}}\Big)\ ds\ B(dy_1)
\ldots B(dy_k)}
\end{equation}
for any $k\geq 1$ and $H>\frac{1}{2}$. It is easy to check that $Z_{H,\lambda}^{k}$ has stationary increment with the scaling property given by
\eqref{eq:scalingTHP}. Moreover, one can verify that $Z_{H,\lambda}^{k}$ has the spectral domain representation
\begin{equation}\label{eq:THPdefharmo orderk}
Z^{k}_{H,\lambda}(t)=c(H,k)\int_{\rr^k}^{''} \frac{e^{it(\omega_1+\ldots+\omega_k)}-1}{i(\omega_1+\ldots+\omega_k)}
\prod_{j=1}^{k}(\lambda+i\omega_j)^{-\big(\frac{1}{2}-\frac{1-H}{k}\big)}
\widehat{B}(d\omega_1)\ldots\widehat{B}(d\omega_k),
\end{equation}
where $c(H,k)=\Big(\frac{\Gamma(\frac{1}{2}-\frac{1-H}{k})}{\sqrt{2\pi}}\Big)^{k}$ is a constant depending on $H$ and $k$. The double prim on the integral indicates that one does not integrate on diagonals where $\omega_i=\omega_j$, $i\neq j$. In this paper, we just consider tempered Hermite process of order one.}}
\end{rem}

Finally, we close this section with introducing tempered Hermite noise which is the increment of tempered Hermite process of order one.
Given a $THP^{1}$ ,\eqref{eq:THPdefn}, we define tempered Hermite noise (THN)
\begin{equation}\label{eq:THNdef}
X_{n}=Z^{1}_{H,\lambda}(n+1)-Z_{H,\lambda}^{1}(n)\quad\text{for integers $0<n<\infty$.}
\end{equation}
It follows easily from \eqref{eq:THPdefn} that THN has the time domain representation
\begin{equation}\label{eq:TFGNmoving}
X_n=
\int_{\rr}\int_{n}^{n+1}(s-y)_{+}^{\frac{3}{2}-H}e^{-\lambda(s-y)_{+}}ds\ B(dy).
\end{equation}
Using \eqref{eq:THPdefharmo}, it also follows that THN has the spectral domain representation,
\begin{equation}\label{eq:THNharmonizable}
X_n=
\frac{1}{C(H)}\int_{\rr}e^{in\omega} \frac{e^{i\omega}-1}{i\omega}({\lambda+i\omega})^{\frac{1}{2}-H}\widehat{B}(d\omega).
\end{equation}
It follows from \eqref{eq:THNharmonizable} that {\rm THN} is a stationary Gaussian time series with mean zero and covariance function
\begin{equation}\label{eq:THNacvf}
r(n):=\mathbb{E}[X_0 X_n]=\frac{\sigma^2}{C(H)^{2}}
\int_{\rr}e^{in\omega}\Big|\frac{e^{i\omega}-1}{i\omega}\Big|^{2}({\lambda^2+\omega^2})^{\frac{1}{2}-H}(d\omega).
\end{equation}

\begin{prop}\label{PropSpedDens}
{\rm THP} \eqref{eq:THNdef} has the spectral density
\begin{equation}\label{eq:specdensTHP}
h(\omega)=\frac{1}{C(H)^2}\Big|\frac{e^{i\omega}-1}{i\omega}\Big|^2
\sum^{+\infty}_{\ell=-\infty} \sigma^2{[\lambda^2+(\omega+2\pi \ell)^2]}^{\frac{1}{2}-H}.
\end{equation}
\end{prop}

\begin{proof}
Recall that the spectral density
\begin{equation}\label{eq:spectralsum}
h(\omega)=\frac{1}{2\pi}\sum^{+\infty}_{j=-\infty}{e^{-i\omega n}}{r(n)}\quad\text{and}\quad r(n)=\int^{\pi}_{-\pi}{e^{i\omega n}}{h(\omega)}d{\omega} .
\end{equation}
Apply \eqref{eq:THNacvf} to write
\begin{equation}\label{specdenscalc}
\begin{split}
r(n)
=&\frac{\sigma^2}{C(H)^2}\int^{+\infty}_{-\infty}e^{i\omega n}\Big|\frac{e^{i\omega}-1}{i\omega}\Big|^{2}(\lambda^2+\omega^2)^{\frac{1}{2}-H}d\omega\\
=&\frac{1}{C(H)^{2}}\int^{+\pi}_{-\pi}e^{i\omega n}\Big|\frac{e^{i\omega}-1}{i\omega}\Big|^{2}
\sum^{+\infty}_{\ell=-\infty}{\sigma^2}{[\lambda^2+(\omega+2\pi \ell)^2]}^{\frac{1}{2}-H}d\omega
\end{split}
\end{equation}
and then it follows from \eqref{eq:spectralsum} that the spectral density of THN is given by \eqref{eq:specdensTHP}.
\end{proof}

\begin{rem}\label{remLowk}
{\emph {Extending the definition \eqref{eq:THNdef} to all $n$ real positive, we obtain the continuous parameter THN
\[X_{t}=Z^{1}_{H,\lambda}(t+1)-Z^{1}_{H,\lambda}(t) .\]
The spectral domain representation of this process is given by \eqref{eq:THNharmonizable} with $n$ replaced by $t$, and the proof of Proposition \ref{PropSpedDens} implies that $X_t$ has spectral density
\begin{equation}\label{eq:specdensTHPcont}
h(\omega)=\frac{\sigma^2}{C(H)^{2}}\Big|\frac{e^{i\omega}-1}{i\omega}\Big|^{2} {[\lambda^2+\omega^2]}^{\frac{1}{2}-H}
\end{equation}
for all real $\omega$.
The fact that $\big|\frac{e^{i\omega}-1}{i\omega}\big|$ is bounded as $\omega\to 0$ yields the low frequency approximation
\begin{equation}\label{eq:spectralowferequncy}
h(\omega)\approx \frac{\sigma^2}{C(H)^{2}}{(\lambda^2+\omega^2)}^{\frac{1}{2}-H}.
\end{equation}
By taking $H=\frac{4}{3}$ in \eqref{eq:spectralowferequncy}, we get $h(\omega)\approx \omega^{-5/3}$ which is the spectral model suggested by Kolmogorov \cite{KolmogorovFBM, Friedlander} for the energy spectrum of turbulence in the inertial range. The spectral density of THN has some applications in turbulent flows \cite{Meerschaertsabzikarkumarzeleki} }}.
\end{rem}

\section{Wiener integrals with respect to tempered Hermite process of order one}\label{sec3}
In order to get the main results of this paper, Section \ref{sec5} , we need to develop the Wiener integrals with respect to $Z^{1}_{H,\lambda}$. We consider two cases:
\begin{itemize}
  \item $\frac{1}{2}<H<1$, $\lambda>0$
  \item $H>1$, $\lambda>0$
\end{itemize}
We start with the first case. We first establish a link between $Z^{1}_{H,\lambda}$ and tempered fractional calculus.
\begin{lem}\label{lem:connection with TFI}
For a tempered Hermite process of order one given by \eqref{eq:THPdefn}, $THP^1$, with $\lambda>0$, we have:
\begin{equation}\label{eq:connection with TFI}
Z^{1}_{H,\lambda}(t)
=\Gamma(H-\frac{1}{2})\int_{-\infty}^{+\infty}\Big({\mathbb{I}}^{H-\frac{1}{2},\lambda}_{-}{\1}_{[0,t]}\Big)(x)\ B(dx)
\end{equation}
where $H>\frac{1}{2}$.
\end{lem}

\begin{proof}
Write the kernel function from \eqref{eq:integrand} in the form
\begin{equation*}
\begin{split}
g_{H,\lambda,t}(x)
&=\int^{t}_{0}(s-x)^{H-\frac{3}{2}}_{+}e^{-\lambda(s-x)_{+}}\ ds\\
&=\int^{+\infty}_{-\infty}{\1_{[0,t]}}(s)(s-x)^{H-\frac{3}{2}}_{+}e^{-\lambda(s-x)_{+}}\ ds\\
&=\Gamma(H-\frac{1}{2})\Big({\mathbb{I}}^{H-\frac{1}{2},\lambda}_{-}{\1}_{[0,t]}\Big)(x)
\end{split}
\end{equation*}
which gives the desired result.
\end{proof}

Next we discuss a general construction for stochastic integrals with respect to $THP^{1}$.  For a standard Brownian motion $\{B(t)\}_{t\in\R}$ on $(\Omega,{\mathcal F},P)$, the stochastic integral ${\Im}(f):=\int f(x)B(dx)$ is defined for any $f\in L^2(\R)$, and the mapping $f\mapsto {\Im}(f)$ defines an isometry from $L^2(\R)$ into $L^2(\Omega)$, called the {\it It\^{o} isometry}:
\begin{equation}\label{eq:ItoIsometry}
\ip{\Im(f)}{\Im(g)}_{L^2(\Omega)}={\rm Cov}[\Im(f),\Im(g)]=\int f(x)g(x)\,dx=\ip fg_{L^2(\R)} .
\end{equation}
Since this isometry maps $L^2(\R)$ onto the space $\overline{\rm Sp}(B)=\{{\Im}(f):f\in L^2(\R)\}$, we say that these two spaces are isometric.  For any elementary function (step function)
\begin{equation}\label{eq:elementarydefn}
f(u)=\sum^{n}_{i=1}a_{i}{\1_{[t_{i},t_{i+1})}(u)},
\end{equation}
where $a_i,t_i$ are real numbers such that $t_i<t_{j}$ for $i<j$, it is natural to define the stochastic integral
\begin{equation}\label{eq:THPintegraldefn}
{\Im}^{\alpha,\lambda}(f)=\int_{\mathbb{R}}f(x)Z^{1}_{H,\lambda}(dx)=\sum^{n}_{i=1}a_i \left[Z^{1}_{H,\lambda}(t_{i+1})-Z^{1}_{H,\lambda}(t_{i})\right],
\end{equation}
and then it follows immediately from \eqref{eq:connection with TFI} that for $f\in {\mathcal E}$, the space of elementary functions, the stochastic integral
\[{\Im}^{\alpha,\lambda}(f)=\int_{\mathbb{R}}f(x)Z^{1}_{H,\lambda}(dx)={\Gamma(H-\frac{1}{2})}\int_{\mathbb{R}}
\Big({\mathbb{I}}^{H-\frac{1}{2},\lambda}_{-}f\Big)(x)\ B(dx)  \]
is a Gaussian random variable with mean zero, such that for any $f,g\in {\mathcal E}$ we have
\begin{equation}\label{eq:innerproductTFI}
\begin{split}
&\ip{{\Im}^{\alpha,\lambda}(f)}{{\Im}^{\alpha,\lambda}(g)}_{L^2(\Omega)} =\mathbb{E}\left(\int_{\mathbb{R}}f(x)Z^{1}_{H,\lambda}(dx)\int_{\mathbb{R}}g(x)Z^{1}_{H,\lambda}(dx)\right) \\
   &=\Gamma(H-\frac{1}{2})^2\int_{\mathbb{R}}
\Big({\mathbb{I}}^{H-\frac{1}{2},\lambda}_{-}f\Big)(x)
\Big({\mathbb{I}}^{H-\frac{1}{2},\lambda}_{-}g\Big)(x)\ dx ,
\end{split}
\end{equation}
in view of \eqref{eq:connection with TFI} and the It\^{o} isometry \eqref{eq:ItoIsometry}.  The linear space of Gaussian random variables $\left\{{\Im}^{\alpha,\lambda}(f),f\in\mathcal{E}\right\}$ is contained in the larger linear space
\begin{equation}
\overline{\rm Sp}(Z^{1}_{H,\lambda})=\left\{X:{\Im}^{\alpha,\lambda}(f_n)\rightarrow X\ \textrm{in $L^2(\Omega)$ for some sequence $(f_n)$ in $\mathcal{E}$}\right\}.
\end{equation}
An element $X\in\overline{\rm Sp}(Z^{1}_{H,\lambda})$ is mean zero Gaussian with variance
\begin{equation*}
\Var(X)=\lim_{n\to\infty} \Var[{\Im}^{\alpha,\lambda}(f_n)] ,
\end{equation*}
and $X$ can be associated with an equivalence class of sequences of elementary functions $(f_n)$ such that ${\Im}^{\alpha,\lambda}(f_n)\to X$ in $L^{2}(\mathbb{R})$. If $[f_X]$ denotes this class, then $X$ can be written in an integral form as
\begin{equation}\label{eq:stochasticintegraldefn}
X=\int_{\mathbb{R}}[f_{X}] dZ^{1}_{H,\lambda}
\end{equation}
and the right hand side of \eqref{eq:stochasticintegraldefn} is called the stochastic integral with respect to $THP^{1}$ on the real line (see, for example, Huang and Cambanis \cite{Huang}, page 587).   In the special case of a Brownian motion $\lambda=0, H=\frac{1}{2}$, ${\Im}^{\alpha,\lambda}(f_n)\rightarrow X$ along with the It\^{o} isometry \eqref {eq:ItoIsometry} implies that $(f_n)$ is a Cauchy sequence, and then since $L^2(\R)$ is a (complete) Hilbert space, there exists a unique $f\in L^2(\R)$ such that $f_n\to f$ in $L^2(\R)$, and we can write $X=\int_{\mathbb{R}} f(x) B(dx)$.  However, if the space of integrands is not complete, then the situation is more complicated.
Here we investigate stochastic integrals with respect to $THP^{1}$ based on time domain representation.
Equation \eqref{eq:innerproductTFI} suggests the appropriate space of integrands for $THP^{1}$, in order to obtain a nice isometry that maps into the space $\overline{\rm Sp}(Z^{1}_{H,\lambda})$ of stochastic integrals.

\begin{thm}\label{thm:stochasticcalculus for TFI and moving general}
Given $\frac{1}{2}<H<1$ and $\lambda>0$, the class of functions
\begin{equation}\label{eq:A1 star class}
{\mathcal{A}}_{1}:=\left\{f\in L^{2}(\mathbb{R}):\int_{\mathbb{R}}
\left|\Big({\mathbb{I}}^{H-\frac{1}{2},\lambda}_{-}f\Big)(x)\right|^{2}dx<\infty\right\},
\end{equation}
is a linear space with inner product
\begin{equation}\label{eq:productTFIf}
\begin{split}
{\langle f,g \rangle}_{{\mathcal{A}}_{1}}
&:={\langle F,G\rangle}_{L^{2}(\mathbb{R})}
\end{split}
\end{equation}
where
\begin{equation}\label{eq:defnF}\begin{split}
F(x)&=\Gamma(H-\frac{1}{2})\Big({\mathbb{I}}^{H-\frac{1}{2},\lambda}_{-}f\Big)(x)\\
G(x)&=\Gamma(H-\frac{1}{2})\Big({\mathbb{I}}^{H-\frac{1}{2},\lambda}_{-}g\Big)(x) .
\end{split}\end{equation}
The set of elementary functions ${\mathcal{E}}$ is dense in the space ${\mathcal{A}}_{1}$.  The space ${\mathcal{A}}_{1}$ is not complete.
\end{thm}

\begin{proof}
The proof is similar to \cite[Theorem 3.5]{MeerschaertsabzikarSPA}. To show that ${\mathcal{A}}_{1}$ is an inner product space, we will check that ${\langle f,f \rangle}_{{\mathcal{A}}_{1}}=0$ implies $f=0$ almost everywhere. If ${\langle f,f \rangle}_{{\mathcal{A}}_{1}}=0$, then in view of \eqref{eq:productTFIf} and \eqref{eq:defnF} we have ${\langle F,F\rangle}_{2}=0$, so $F(x)=\Gamma(H-\frac{1}{2})\Big({\mathbb{I}}^{H-\frac{1}{2},\lambda}_{-}f\Big)(x)=0$ for almost every $x\in\R$. Then
\begin{equation}\label{eq1:stochasticcalculus for TFI and moving general}
\Big({\mathbb{I}}^{H-\frac{1}{2},\lambda}_{-}f\Big)(x)=0\quad\text{for almost every $x\in\R$.}
\end{equation}
Apply ${\mathbb{D}}^{H-\frac{1}{2},\lambda}_{-}$ to both sides of equation \eqref{eq1:stochasticcalculus for TFI and moving general} and use Lemma \ref{lem:inversoperator} to get $f(x)=0$ for almost every $x\in\R$, and hence ${\mathcal{A}}_{1}$ is an inner product space.

Next, we want to show that the set of elementary functions ${\mathcal E}$ is dense in ${\mathcal{A}}_{1}$. For any $f\in{\mathcal{A}}_{1}$, we also have $f\in{L}^{2}(\mathbb{R})$, and hence there exists a sequence of elementary functions $(f_n)$ in $L^2(\R)$ such that $\|f-f_n\|_2\to 0$.
But
\begin{equation*}
\|f-f_n\|_{{\mathcal{A}}_{1}}={\langle f-f_n,f-f_n \rangle}_{{\mathcal{A}}_{1}}={\langle F-F_n,F-F_n\rangle}_2=\|F-F_n\|_2,
\end{equation*}
where
$
F_n(x)=\Big({\mathbb{I}}^{H-\frac{1}{2},\lambda}_{-}{f_n}\Big)(x)
$
and $F(x)$ is given by \eqref{eq:defnF}.  Lemma \ref{lem:TFI and Lp} implies that
\begin{equation*}
\|f-f_n\|_{{\mathcal{A}}_{1}}=\left\|F-F_n\right\|_{2}=\|{\mathbb{I}}^{H-\frac{1}{2},\lambda}_{-}(f-f_n)\|_{2}\leq C\|f-f_n\|_{2}
\end{equation*}
for some $C>0$, and since $\|f-f_n\|_2\to 0$, it follows that the set of elementary functions is dense in ${{\mathcal{A}}_{1}}$.

Finally, we provide an example to show that ${\mathcal{A}}_{1}$ is not complete. The functions
\begin{equation*}
\widehat{f_n}(\omega)=|\omega|^{-p}\1_{\{1<|\omega|<n\}}(\omega),\ p>0,
\end{equation*}
are in $L^{2}(\mathbb{R})$, $\overline{\widehat{f_n}(\omega)}=\widehat{f_n}(-\omega)$, and hence they are the Fourier transforms of functions $f_n\in L^{2}(\mathbb{R})$.  Apply Lemma \ref{lem:FourierTFI} to see that the corresponding functions $F_n(x)=\Gamma(H-\frac{1}{2})\Big({\mathbb{I}}^{H-\frac{1}{2},\lambda}_{-}f_n\Big)(x)$ from \eqref{eq:defnF} have Fourier transform
\begin{equation}\begin{split}\label{eq:A1Fdef}
{\mathcal{F}}[F_n](\omega)
&=\Gamma(H-\frac{1}{2})(\lambda+i\omega)^{\frac{1}{2}-H}\hat f_n(\omega).
\end{split}\end{equation}
Since $\frac{1}{2}-H<0$, it follows that
\[\|F_n\|_2^2=\|\hat F_n\|_2^2=\Gamma(H-\frac{1}{2})^2\int_{-\infty}^{\infty}\left|\widehat{f_n}(\omega)\right|^{2}(\lambda^2+\omega^2)^{\frac{1}{2}-H}<\infty\]
for each $n$, which shows that $f_n\in {\mathcal{A}}_{1}$.  Now it is easy to check that $f_n-f_m\to 0$ in ${{\mathcal{A}}_{1}}$, as $n,m\to\infty$,
whenever  $p>1-H$, so that $(f_n)$ is a Cauchy sequence.  Choose $p=\frac{1}{2}$ and suppose that there exists some $f\in{{\mathcal{A}}_{1}}$ such that $\|f_n- f\|_{{\mathcal{A}}_{1}}\to 0$ as $n\to\infty$. Then
\begin{equation}\label{eq:a3}
\int_{-\infty}^{\infty}\left|\widehat{f_n}(\omega)-\widehat{f}(\omega)\right|^{2}(\lambda^2+\omega^2)^{\frac{1}{2}-H}\to 0
\end{equation}
as $n\to\infty$, and since, for any given $m\geq 1$, the value of $\widehat{f_n}(\omega)$ does not vary with $n>m$ whenever $\omega\in [-m,m]$, it follows that $\hat{f}(\omega)=|\omega|^{-\frac{1}{2}}1_{\{|\omega|>1\}}$ on any such interval. Since $m$ is arbitrary, it follows that $\hat{f}(\omega)=|\omega|^{-\frac{1}{2}}1_{\{|\omega|>1\}}$, but this function is not in $L^{2}(\mathbb{R})$, so $\hat{f}(\omega)\notin {{\mathcal{A}}_{1}}$, which is a contradiction.  Hence ${{\mathcal{A}}_{1}}$ is not complete, and this completes the proof.
\end{proof}

\begin{rem}
{\emph{It follows from Lemma \ref{lem:TFI and Lp} that ${\mathcal{A}}_{1}$ contains every function in $L^{2}(\mathbb{R})$, and hence they are the same set, but endowed with a different inner product.}}
\end{rem}

We now define the stochastic integral with respect to $THP^{1}$ for any function in ${\mathcal{A}}_{1}$ in the case where $\frac{1}{2}<H<1$.
\begin{defn}\label{defn:TFI of general f}
For any $\frac{1}{2}<H<1$ and $\lambda>0$, we define
\begin{equation}\label{eq:TFIof f resp $THP^{1}$ genral}
\int_{\mathbb{R}}f(x)Z^{1}_{H,\lambda}(dx):={\Gamma(H-\frac{1}{2})}\int_{\mathbb{R}}
\Big({\mathbb{I}}^{H-\frac{1}{2},\lambda}_{-}f\Big)(x)\ B(dx)
\end{equation}
for any $f\in{\mathcal{A}}_{1}$.
\end{defn}

\begin{thm}\label{thm:SLRDisometric}
For any $\frac{1}{2}<H<1$ and $\lambda>0$, the stochastic integral $\Im^{\alpha,\lambda}$ in \eqref{eq:TFIof f resp $THP^{1}$ genral} is an isometry from  ${\mathcal{A}}_{1}$ into $\overline{\rm Sp}(Z^{1}_{H,\lambda})$.  Since ${\mathcal{A}}_{1}$ is not complete, these two spaces are not isometric.
\end{thm}

\begin{proof}
It follows from Lemma \ref{lem:TFI and Lp} that the stochastic integral \eqref{eq:TFIof f resp $THP^{1}$ genral} is well-defined for any $f\in{\mathcal{A}}_{1}$.  Proposition 2.1 in Pipiras and Taqqu \cite{PipirasTaqqu} implies that, if $\mathcal{D}$ is an inner product space such that $(f,g)_{\mathcal{D}}=\ip{{\Im}^{\alpha,\lambda}(f)}{{\Im}^{\alpha,\lambda}(g)}_{L^2(\Omega)}$ for all $f,g\in\mathcal{E}$, and if $\mathcal{E}$ is dense $\mathcal{D}$, then there is an isometry between $\mathcal{D}$  and a linear subspace of $\overline{\rm Sp}(Z^{1}_{H,\lambda})$ that extends the map $f\to{\Im}^{\alpha,\lambda}(f)$ for $f\in\mathcal{E}$, and furthermore, $\mathcal{D}$ is isometric to $\overline{\rm Sp}(Z^{1}_{H,\lambda})$ itself if and only if $\mathcal{D}$ is complete.  Using the It\^{o} isometry and the definition \eqref{eq:TFIof f resp $THP^{1}$ genral}, it follows from \eqref{eq:productTFIf} that for any $f,g\in{{\mathcal{A}}_{1}}$ we have
 \[{\langle f,g \rangle}_{{\mathcal{A}}_{1}}={\langle F,G\rangle}_{L^{2}(\mathbb{R})}
 =\ip{{\Im}^{\alpha,\lambda}(f)}{{\Im}^{\alpha,\lambda}(g)}_{L^2(\Omega)} ,\]
and then the result follows from Theorem \ref{thm:stochasticcalculus for TFI and moving general}.
\end{proof}

We now apply the spectral domain representation of $THP^{1}$ to investigate the stochastic integral with respect to $THP^{1}$.
Apply the Fourier transform of an indicator function to write this spectral domain representation in the form
\begin{equation*}\label{eq:$THP^{1}$defharmo and indicator*}
Z^{1}_{H,\lambda}(t)=\Gamma(H-\frac{1}{2})\int^{+\infty}_{-\infty} \widehat{\1}_{[0,t]}(\omega)(\lambda+i\omega)^{\frac{1}{2}-H}\hat{B}(d\omega).
\end{equation*}
It follows easily that for any elementary function \eqref{eq:elementarydefn} we may write
\begin{equation}\label{eq:$THP^{1}$defharmo and elementary}
\Im^{\alpha,\lambda}(f)=\Gamma(H-\frac{1}{2})\int^{\infty}_{-\infty} \widehat{f}(\omega)(\lambda+i\omega)^{\frac{1}{2}-H}\hat{B}(d\omega) ,
\end{equation}
and then for any elementary functions $f$ and $g$ we have
\begin{equation}\label{eq:inner harmo}
\ip{\Im^{\alpha,\lambda}(f)}{\Im^{\alpha,\lambda}(g)}_{L^2(\Omega)} =\Gamma(H-\frac{1}{2})
   \int_{-\infty}^{\infty}\widehat{f}(\omega)\overline{\widehat{g}(\omega)}(\lambda^2+\omega^2)^{\frac{1}{2}-H}d\omega .
\end{equation}

\begin{thm}\label{thm:A2space}
For any $\frac{1}{2}<H<1$ and $\lambda>0$, the class of functions
\begin{equation}\label{eq:A3class}
\mathcal{A}_{2}:=\left\{f\in L^{2}(\mathbb{R}):\int \left|\widehat{f}(\omega)\right|^{2}(\lambda^2+\omega^2)^{\frac{1}{2}-H}\ d\omega<\infty\right\} ,
\end{equation}
is a linear space with the inner product
\begin{equation}\label{eq:productharmo}
{\langle f,g \rangle}_{{\mathcal{A}}_{2}}=\Gamma(H-\frac{1}{2})^2
   \int_{-\infty}^{+\infty}\widehat{f}(\omega)\overline{\widehat{g}(\omega)}(\lambda^2+\omega^2)^{\frac{1}{2}-H}d\omega .
\end{equation}
The set of elementary functions ${\mathcal{E}}$ is dense in the space ${\mathcal{A}}_2$.  The space ${\mathcal{A}}_2$ is not complete.
\end{thm}

\begin{proof}
The proof combines Theorem \ref{thm:stochasticcalculus for TFI and moving general} and using the Plancherel Theorem.

Since $H>\frac{1}{2}$, the function $(\lambda^2+\omega^2)^{\frac{1}{2}-H}$ is bounded by a constant $C(H,\lambda)$ that depends only on $H$ and $\lambda$, so for any $f\in L^{2}(\mathbb{R})$ we have
\begin{equation}\label{eq:A3toL2}
\int_{\mathbb{R}}
\left|\widehat{f}(\omega)\right|^{2}(\lambda^2+\omega^2)^{\frac{1}{2}-H}\ d\omega\leq C(H,\lambda) \int_{\mathbb{R}}
\left|\widehat{f}(\omega)\right|^{2}\ d\omega <\infty
\end{equation}
and hence $f\in \mathcal{A}_{2}$. Since $\mathcal{A}_{2}\subset L^{2}(\mathbb{R})$ by definition, this proves that $L^{2}(\mathbb{R})$ and $\mathcal{A}_{2}$ are the same set of functions, and then it follows from Lemma \ref{lem:TFI and Lp} that $\mathcal{A}_{1}$ and $\mathcal{A}_{2}$ are the same set of functions. Observe that $\varphi_f=\Big({\mathbb{I}}^{H-\frac{1}{2},\lambda}_{-}f\Big)$ is again a function with Fourier transform
\[\hat \varphi_f=(\lambda+i\omega)^{\frac{1}{2}-H}\hat f .\]
Then it follows from the Plancherel Theorem that
\begin{equation*}\begin{split}
\ip fg_{{\mathcal{A}}_{1}}=\Gamma(H-\frac{1}{2})^2\ip {\varphi_f}{\varphi_g}_2&=\Gamma(1-\alpha)^2\ip {\hat\varphi_f}{\hat\varphi_g}_2\\
&=\Gamma(1-\alpha)^2\int_{-\infty}^{+\infty}\widehat{f}(\omega)\overline{\widehat{g}(\omega)}(\lambda^2+\omega^2)^{\frac{1}{2}-H}d\omega=\ip fg_{{\mathcal{A}}_{2}} \end{split}\end{equation*}
and hence the two inner products are identical.  Then the conclusions of Theorem \ref{thm:A2space} follow from Theorem \ref{thm:stochasticcalculus for TFI and moving general}.
\end{proof}

\begin{defn}\label{defn:stochIntFourier}
For any $H>\frac{1}{2}$ and $\lambda>0$, we define
\begin{equation}\label{eq:stochIntFourier}
\Im^{\alpha,\lambda}(f)=\Gamma(1-\alpha)\int^{\infty}_{-\infty} \widehat{f}(\omega)(\lambda+i\omega)^{\frac{1}{2}-H}\hat{B}(d\omega)
\end{equation}
for any $f\in{\mathcal{A}}_{2}$.
\end{defn}

\begin{thm}\label{thm:SLRDisometric2}
For any $H>\frac{1}{2}$ and $\lambda>0$, the stochastic integral $\Im^{\alpha,\lambda}$ in \eqref{eq:stochIntFourier} is an isometry from  ${\mathcal{A}}_{2}$ into $\overline{\rm Sp}(Z^{1}_{H,\lambda})$.  Since ${\mathcal{A}}_{2}$ is not complete, these two spaces are not isometric.
\end{thm}

\begin{proof}
The proof of Theorem \ref{thm:A2space} shows that $\mathcal{A}_{1}$ and $\mathcal{A}_{2}$ are identical when $H>\frac{1}{2}$. Then the result follows immediately from Theorems \ref{thm:SLRDisometric}.
\end{proof}

Now, we consider the second case that we mentioned at the beginning of this section.  we will show that $Z^{1}_{H,\lambda}$ is a continuous semimartingale with finite variation and hence one can define stochastic integrals $I(f):=\int f(x)Z^{1}_{H,\lambda}(dx)$ in the standard manner, via the It\^{o} stochastic calculus (e.g., see Kallenberg \cite[Chapter 15]{Kallenberg}).
\begin{thm}\label{thm:THPsemimartingale}
A tempered Hermite process of order one $\{Z^{1}_{H,\lambda}(t)\}_{t\geq 0}$ with $H>1$ and $\lambda>0$ is a continuous semimartingale with the canonical decomposition
\begin{equation}\label{thm:THPsemimartingaleEq}
Z^{1}_{H,\lambda}(t)=\int_{0}^{t} M_{H,\lambda}(s)\ ds
\end{equation}
where
\begin{equation}\label{eq:THPdef2}
M_{H,\lambda}(s):={\int^{+\infty}_{-\infty}(s-y)_{+}^{H-\frac{3}{2}} e^{-\lambda(s-y)_{+}}\ B(dy)} .
\end{equation}
Moreover, $\{Z^{1}_{H,\lambda}(t)\}_{t\geq 0}$ is a finite variation process.
\end{thm}

\begin{proof}
Let $\{{\mathcal{F}}^{B}_{t}\}_{t\geq 0}$ be the $\sigma$-algebra generated by $\{B_s:0\leq s\leq t\}$.  Given a function  $g:{\mathbb R}\to {\mathbb R}$ such that $g(t)=0$ for all $t<0$,
and
\begin{equation}\label{gtCheridito}
g(t)=C+\int_{0}^{t}h(s)\ ds,\quad\text{for all $t>0$},
\end{equation}
for some $C\in\mathbb{R}$ and some $h\in L^{2}(\mathbb{R})$, a result of Cheridito \cite[Theorem 3.9]{Cheridito} shows that the Gaussian stationary increment process
\begin{equation}\label{CheriditoDef}
Y^{g}_{t}:=\int_{\mathbb{R}}[g(t-y)-g(-y)]\ B(dy),\ t\geq 0
\end{equation}
is a continuous $\{{\mathcal{F}}^{B}_{t}\}_{t\geq 0}$ semimartingale with canonical decomposition
\begin{equation}
Y^{g}_{t}=g(0)B_{t}+\int_{0}^{t}\int_{-\infty}^{s}h(s-y)B(dy)ds ,
\end{equation}
and conversely, that if \eqref{CheriditoDef} defines a semimartingale on $[0,T]$ for some $T>0$, then $g$ satisfies these properties.
Define $g(t)=0$ for $t\leq 0$ and
\begin{equation}
g(t):=\int_{0}^{t}s^{H-\frac{3}{2}}e^{-\lambda s}\ ds\quad\text{for $t>0$.}
\end{equation}
It is easy to check that the function $g(t-y)-g(-y)$ is square integrable over the entire real line for any $H>\frac{1}{2}$ and $\lambda>0$ (See Lemma \ref{lem:g squar integrable}). Next observe that \eqref{gtCheridito} holds with $C=0$, $h(s)=0$ for $s<0$ and
\begin{equation}
h(s):=s^{H-\frac{3}{2}}e^{-\lambda s}\in L^{2}(\mathbb{R})
\end{equation}
for any $H>1$ and $\lambda>0$.  Then it follows from \cite[Theorem 3.9]{Cheridito} that $THP^{1}$ is a continuous semimartingale with canonical decomposition
\begin{equation}
\begin{split}
Z^{1}_{H,\lambda}&=\int_{-\infty}^{+\infty}\int_{0}^{t}(s-y)_{+}^{H-\frac{3}{2}} e^{-\lambda(s-y)_{+}}\ ds\ B(dy)\\
&=\int_{0}^{t}\int_{-\infty}^{+\infty}(s-y)_{+}^{H-\frac{3}{2}} e^{-\lambda(s-y)_{+}} B(dy)\ ds\\
\end{split}
\end{equation}
which reduces to \eqref{thm:THPsemimartingaleEq}. Since $C=0$, Theorem 3.9 in \cite{Cheridito} implies that $\{Z^{1}_{H,\lambda}(t)\}$ is a finite variation process.
\end{proof}

\begin{rem}
{\emph{ When $H=\frac{3}{2}$ and $\lambda>0$, the Gaussian stochastic process \eqref{eq:THPdef2} is an Ornstein-Uhlenbeck process.  When $H>1$ and $\lambda>0$, it is a one dimensional Mat\'ern stochastic process \cite{Banerjee,Gneiting,Handcock}, also called a ``fractional Ornstein-Uhlenbeck process''  in the physics literature \cite{lim}.  It follows from Knight \cite[Theorem 6.5]{Knight} that $M_{H,\lambda}(t)$ is a semimartingale in both cases.}}
\end{rem}

Cheridito \cite[Theorem 3.9]{Cheridito} provides a necessary and sufficient condition for the process \eqref{CheriditoDef} to be a semimartingale, and then it is not hard to check that $THP^{1}$ is {\em not a semimartingale} in the remaining case when $\frac{1}{2}<H<1$.

\section{ARTFIMA time series; Definition and basic properties}\label{sec4}
In this section, we first recall the definition of the autoregressive tempered fractionally integrated moving average (ARTFIMA) time series and some of its basic properties such as the covariance  function and spectral density.

The tempered fractional difference operator is defined by:
\begin{equation}\label{eq:TFdiffDef}
\Delta^{\alpha,\lambda}_h f(x)=\sum_{j=0}^\infty w_j e^{-\lambda jh}f(x-jh) \quad\text{with}\quad  w_j:=(-1)^j\binom \alpha j=\frac{(-1)^j\Gamma(1+\alpha)}{j!\Gamma(1+\alpha-j)}
\end{equation}
for $\alpha>0$ and $\lambda>0$, where $\Gamma(\cdot)$ is the Euler gamma function.
If $\lambda=0$ and $\alpha$ is a positive integer, then equation \eqref{eq:TFdiffDef} reduces to the usual definition of the fractional difference operator.

The ARMA$(p,q)$ model, which combines an autoregression of order $p$ with a moving average of order $q$, is defined by
\begin{equation}\label{eq:ARMA}
X_t-\sum_{j=1}^p \phi_jX_{t-j}=Z_t+\sum_{i=1}^q \theta_i Z_{t-i}
\end{equation}
where $\{Z_t\}$ is an i.i.d.\  sequence of uncorrelated random variables (white noise). We now recall the definition of the ARTFIMA $(p,\alpha,\lambda,q)$.
\begin{defn}
The discrete time stochastic process $\{X_t\}$ is called an {\it autoregressive tempered fractional integrated moving average} , ${ARTFIMA}\ (p,\lambda,\alpha,q)$, if
\begin{equation}
\Delta^{\alpha,\lambda}_{1}X{_t}=\sum_{j=0}^\infty w_j e^{-\lambda jh}X_{t-j},
\end{equation}
follows an ARMA$(p,q)$ model.
\end{defn}
Let $\{X_t\}$ be an ARTFIMA $(0,\lambda,\alpha,0)$ process. Then,
\begin{equation*}
X_t=\Delta^{-\alpha,\lambda}_1 Z_{t}=\sum_{j=0}^{\infty}(-1)^j e^{-\lambda j}\binom{-\alpha}{j}Z_{t-j},\\
\end{equation*}
where $\Delta^{-\alpha,\lambda}_1$ is the inverse operator of $\Delta^{\alpha,\lambda}_1$ and can be defined by \eqref{eq:TFdiffDef}. In other words, $X_t=\Delta^{-\alpha,\lambda}_1 Z_t$, is a tempered fractionally integrated ARMA$(p,q)$ model.  The fractional integration operator $\Delta^{-\alpha,\lambda}_1$, the inverse of $\Delta^{\alpha,\lambda}_1$, is also defined by \eqref{eq:TFdiffDef}. We refer the reader to \cite{TFC} to find more details about the tempered fractional difference operator. In this paper we are interested in the case $ARTFIMA\ (0,\alpha,\lambda,0)$.
\begin{rem}
{\emph {Since $\{Z_t\}$ is stationary and \begin{equation*}
\sum_{j=0}^{\infty}(-1)^j e^{-\lambda j}\binom{-\alpha}{j}=(1-e^{-\lambda})^{-\alpha}<\infty
\end{equation*}
for any $\alpha>0$ and $\lambda>0$, Proposition 3.1.2 in \cite{BrockwellDavisTSTM} implies that the series
\begin{equation*}
X_t=\Delta^{-\alpha,\lambda}_1 Z_{t}=\sum_{j=0}^{\infty}(-1)^j e^{-\lambda j}\binom{-\alpha}{j}Z_{t-j}
\end{equation*}
is stationary and converges absolutely with probability one.}}
\end{rem}
\begin{rem}
{\emph{Peiris \cite{Peiris} has proposed a generalized autoregressive GAR$(p)$ time series model $(1-\beta B)^\alpha X_t=Z_t$ for applications in finance, where $|\beta|<1$.  Taking $\beta=e^{-\lambda}$ we obtain the ARTFIMA$(0,\alpha,\lambda,0)$ model.}}
\end{rem}
We next state the spectral density and covariance function of ARTFIMA $(0,\alpha,\lambda,0)$.
\begin{thm}\label{thm:Proprties}
Let $\{X_t\}$ be an ARTFIMA $(0,\alpha,\lambda,0)$ times series.
\begin{description}
\item  [a] $\{X_t\}$ has the spectral density
  \begin{equation}\label{eq:spectral}
   h(\omega)=\frac{\sigma^2}{2\pi}\Big|1-e^{-(\lambda+i\omega)}\Big|^{-2\alpha},
  \end{equation}
  for $-\pi\leq \omega\leq \pi$.
  \item [b] The covariance function of $\{X_t\}$ is
\begin{equation}\label{eq:covarianceartfima}
\gamma_k=\mathbb{E}(X_t X_{t+k})=\frac{\sigma^2}{2\pi}\frac{e^{-\lambda k}\Gamma(\alpha+k)}{\Gamma(\alpha)k!}\ {_2F_1(\alpha;k+\alpha;k+1;e^{-2\lambda})},
\end{equation}
where $_2F_1(a;b;c;z)=\sum_{j=0}^{\infty}\frac{\Gamma(a+j)\Gamma(b+j)\Gamma(c)z^j}{\Gamma(a)\Gamma(b)\Gamma(c+j)\Gamma(j+1)}$ is the hypergeometric function.
\end{description}
\end{thm}
\begin{proof}
(a) Writing $X_t=\psi_{\lambda}(B)Z_t$, we have $\psi_{\lambda}(B)=(1-e^{-\lambda}B)^{-\alpha}$. Then the general theory of linear filters implies that $X_t$ has spectral density $f_X(k)=|\Psi(e^{-ik})|^2f_Z(k)$ using the complex absolute value (e.g., see  \cite{BrockwellDavisTSTM}). That is
\begin{equation*}
\begin{split}
h(\omega)&=\frac{\sigma^2}{2\pi}\psi_{\lambda}(e^{i\omega})\psi_{\lambda}(e^{-i\omega})\\
&=\frac{\sigma^2}{2\pi}\left(1-2e^{-\lambda}\cos{\omega}+e^{-2\lambda}\right)^{-\alpha}\\
&=\Big|1-e^{-(\lambda+i\omega)}\Big|^{-2\alpha}
\end{split}
\end{equation*}
and this gives \eqref{eq:spectral}. In order to show $(b)$, we have
\begin{equation*}
\begin{split}
\gamma_k &=\int_{-\pi}^{\pi}\cos{(k\omega)}h(\omega)\ d\omega\\
&=\int_{-\pi}^{\pi}\frac{\sigma^2}{2\pi}\frac{\cos{(k\omega)}}{\left(1-2e^{-\lambda}\cos{\omega}+e^{-2\lambda}\right)^{\alpha}}\ d\omega\\
&=\int_{0}^{2\pi}\frac{\sigma^2}{2\pi}\frac{(-1)^{k}\cos{(k\omega')}}{\left(1-2e^{-\lambda}\cos{\omega}+e^{-2\lambda}\right)^{\alpha}}\ d\omega'\ [\omega':=\omega+\pi]\\
&=\sigma^2\frac{e^{-\lambda k}\Gamma(k+\alpha)}{\Gamma(\alpha)k!}\ {_2F_1(\alpha;k+\alpha;k+1;e^{-2\lambda})},
\end{split}
\end{equation*}
where we applied the following integral formula:
\begin{equation*}
\frac{1}{2\pi}\int_{0}^{2\pi}\frac{\cos{k\omega'}}{\left(1-2z\cos{\omega}+z^2\right)^{\alpha}}\ d\omega'
=\frac{z^k \Gamma(k-\alpha)}{\Gamma(\alpha)k!}{_2F_1(\alpha;k+\alpha;k+1;z^2)}
\end{equation*}
and hence we proved part (b).
\end{proof}

The next lemma gives the spectral representation of the ARTFIMA $(0,\alpha,\lambda,0)$. We will use this lemma in the next section.
\begin{lem}\label{lem:spectralrepresntation ARTFIMA}
Let $X^{\lambda}_{k}$ be the ARTFIMA $(0,\alpha,\alpha,0)$ time series such that $X^{\lambda}_{k}=\sum_{j=0}^{\infty}(-1)^j e^{-\lambda j}\binom{-\alpha}{j}Z_{t-j}$. Then, $X^{\lambda}_{k}$
has the spectral representation
\begin{equation*}
X^{\lambda}_{k}=\int_{-\pi}^{\pi}e^{ik\nu}\ dW_{\lambda}(\nu),
\end{equation*}
where $dW_{\lambda}(\nu)=\Big(1-e^{-(\lambda+i\nu)}\Big)^{-\alpha}dW$ and $\{W(\nu), -\pi\leq \nu\leq\pi\}$ is a right-continuous orthogonal increment process.
\end{lem}
\begin{proof}
Suppose $\{Z_t\}$ has the spectral representation $Z_t=\int_{-\pi}^{\pi}e^{ik\nu}dW(\nu)$, where $\{W(\nu), -\pi\leq \nu\leq\pi\}$ is a right-continuous orthogonal increment process. Then Theorem 4.10.1 in \cite{BrockwellDavisTSTM} implies that $X^{\lambda}_{k}$ has the spectral representation
\begin{equation*}
\begin{split}
X^{\lambda}_{k}&=\int_{-\pi}^{\pi}e^{ik\nu}\sum_{j=0}^{\infty}(-1)^j e^{-\lambda j}\binom{-\alpha}{j}e^{-ij\nu}\ dW(\nu)\\
&=\int_{-\pi}^{\pi}e^{ik\nu}\Big(1-e^{-(\lambda+i\nu)}\Big)^{-\alpha}\ dW(\nu),
\end{split}
\end{equation*}
and this completes the proof.
\end{proof}

The next lemma gives the asymptotic result of the covariance function of the ARTFIMA $(0,\alpha,\lambda,0)$.
\begin{lem}\label{lem:asymARTFIMA}
Let $\{X_t\}$ be an ARTFIMA $(0,\alpha,\lambda,0)$ times series with the covariance function $\gamma_k=\mathbb{E}(X_t X_{t+k})$ given by \eqref{eq:covarianceartfima}. Then
\begin{equation*}
\gamma_k\sim \frac{\sigma^2}{2\pi}\frac{1}{\Gamma(\alpha)}e^{-\lambda k}k^{\alpha-1}(1-e^{-2\lambda})^{-\alpha}
\end{equation*}
as $|k|\to\infty$.
\end{lem}
\begin{proof}
By applying \eqref{eq:covarianceartfima} and the fact that $\frac{\Gamma(a+k)}{\Gamma(b+k)}\sim k^{a-b}$, as $k\to\infty$, we have
\begin{equation*}
\begin{split}
\gamma_k&=\sigma^2\frac{e^{-\lambda k}\Gamma(k+\alpha)}{\Gamma(\alpha)k!}\ {_2F_1(\alpha;k+\alpha;k+1;e^{-2\lambda})}\\
&=\sigma^2\frac{e^{-\lambda k}\Gamma(k+\alpha)}{\Gamma(\alpha)k!}\sum_{j=0}^{\infty}\frac{\Gamma(\alpha+j)\Gamma(k+\alpha+j)\Gamma(k+1)e^{-2\lambda j}}{\Gamma(\alpha)\Gamma(k+\alpha)\Gamma(k+j+1)(j)!}\\
& \sim \sigma^2\frac{1}{\Gamma(\alpha)}e^{-\lambda k}k^{\alpha-1}
\sum_{j=0}^{\infty}\frac{\Gamma(\alpha+j)e^{-2\lambda j}}{\Gamma(\alpha)(j)!}\\
&=\sigma^2\frac{1}{\Gamma(\alpha)}e^{-\lambda k}k^{\alpha-1}(1-e^{-2\lambda})^{-\alpha},
\end{split}
\end{equation*}
which gives the desired result.
\end{proof}
\begin{rem}
{\emph{ The ARTFIMA $(0,\alpha,\lambda,0)$ is short memory process, since by Lemma \ref{lem:asymARTFIMA} one can show that $\sum_{k=0}^{\infty}\gamma_k<\infty$. }}
\end{rem}

\section{Weak Convergence Results}\label{sec5}

We now in a position to answer the first and second question. We start with the first one.

Assume $H=\frac{1}{2}+\alpha$ for $\alpha>0$ and let $\{Z_j\}_{j\in\mathbb{Z}}$ be a sequence of independent and identically distributed random variables mean zero and variance one .
Define the random variables
\begin{equation}\label{eq:X definition}
Y^{\frac{\lambda}{n}}_{k}:=\sum_{j\in\mathbb{Z}}C^{\frac{\lambda}{n}}_{j} Z_{k-j},\quad k=1,2,\ldots
\end{equation}
where
\begin{equation}\label{eq:Cdef}
C^{\frac{\lambda}{n}}_{j}= \begin{cases}\frac{1}{\Gamma(\alpha)} j^{\alpha-1}e^{-\frac{\lambda}{n}j}  &\mbox{if } j\geq 1 \\
0 & \mbox{if } j\leq 0. \end{cases}
\end{equation}
For $t\geq 0$, we define
$S^{\frac{\lambda}{n}}(t)$ as the partial sum of $\{X^{\frac{\lambda}{n}}_{k}\}$,
\begin{equation}\label{eq:S definition}
S^{\frac{\lambda}{n}}(t):=\sum_{k=1}^{[t]}Y^{\frac{\lambda}{n}}_{k}+(t-[t])Y^{\frac{\lambda}{n}}_{[t]+1},\quad t\geq 0,
\end{equation}
where $[t]$ is the largest integer less than or equals $t$ and $\sum_{k=1}^{0}=0$. We also define
\begin{equation}\label{eq:Xi definition}
\xi^{\frac{\lambda}{n}}_{m}(t):=\sum_{j=1-m}^{[t]-m}C^{\frac{\lambda}{n}}_{j}+(t-[t])C^{\frac{\lambda}{n}}_{t+1-m},
\end{equation}
for $m\in\mathbb{Z}$ and $t\geq 0$, where $\sum_{j=1-m}^{-m}=0$. Then we have from \eqref{eq:X definition} and \eqref{eq:S definition},
\begin{equation}\label{eq:S and Xi connection}
S^{\frac{\lambda}{n}}(t)=\sum_{m\in\mathbb{Z}}\xi^{\frac{\lambda}{n}}_{m}(t) Z_m.
\end{equation}
On the other hand,
\begin{equation}\label{eq:Xi and coefficient connection}
\begin{split}
\xi^{\frac{\lambda}{n}}_{m}(nt)&=\sum_{j=1-m}^{[nt]-m}C^{\frac{\lambda}{n}}_{j}
\sim \int_{-m}^{[nt]-m}j^{\alpha-1}e^{-\frac{\lambda}{n}j}\ dj\\
&=\Big(\frac{n}{\lambda}\Big)^{\alpha}\int_{-\frac{\lambda m}{n}}^{\frac{\lambda }{n}([nt]-m)}{\omega}^{\alpha-1}e^{-\omega}\ d\omega\\
&=\Big(\frac{n}{\lambda}\Big)^{\alpha}\Big[\gamma(\alpha,\frac{\lambda }{n}([nt]-m))-\gamma(-\frac{\lambda m}{n})\Big],
\end{split}
\end{equation}
when $m$ is negative and $|m|$ is large. From \eqref{eq:S and Xi connection} and \eqref{eq:Xi and coefficient connection} we have:
\begin{lem}
For any $\theta_1,\theta_2,\ldots,\theta_p$, $t_1,t_2,\ldots,t_p\geq 0$, we have
\begin{equation*}
n^{-2H}\sum_{m\in\mathbb{Z}}\Big|\sum_{r=1}^{p}\theta_r \xi^{\frac{\lambda}{n}}_{m}(nt)\Big|^{2}
\rightarrow \int_{-\infty}^{+\infty}\Big|\sum_{r=1}^{p}\theta_r
\int_{0}^{t}(s-y)_{+}^{H-\frac{3}{2}}e^{-\lambda(s-y)_{+}}\ ds \Big|^{2}\ dy
\end{equation*}
as $n\to \infty$.
\end{lem}
\begin{proof}
By applying \eqref{eq:Xi and coefficient connection}, we get:
\begin{equation*}
\begin{split}
&n^{-2H}\sum_{m\in\mathbb{Z}}\Big|\sum_{r=1}^{p}\theta_{r}\xi^{\frac{\lambda}{n}}_{m}(nt_r) \Big|^2\\
&\sim n^{-2H}\Big(\frac{n}{\lambda}\Big)^{2\alpha}\sum_{m\in\mathbb{Z}}\Big|\sum_{r=1}^{p}\theta_{r}\Big[
\gamma(\alpha,\frac{\lambda }{n}([nt]-m))-\gamma(\alpha,-\frac{\lambda m}{n})\Big]\Big|^2\\
&=\Big(\frac{1}{\lambda}\Big)^{2\alpha}n^{2\alpha-2H+1}\ n^{-1}\sum_{m\in\mathbb{Z}}\Big|\sum_{r=1}^{p}\theta_{r}\Big[
\gamma(\alpha,\frac{\lambda }{n}([nt]-m))-\gamma(-\frac{\lambda m}{n})\Big]\Big|^2\\
&\rightarrow \Big(\frac{1}{\lambda}\Big)^{2\alpha}\int_{\rr}\Big|\sum_{r=1}^{p}\theta_{r}\int_{-\lambda y}^{\lambda t_r- \lambda y}
x_{+}^{\alpha-1}e^{-(x)_{+}}\ dx\Big|^{2}\ dy,\ ({\textrm{as}}\quad n\to\infty),\\
&\qquad\qquad ({\textrm{define}}\quad \lambda s=x+\lambda y)\\
&=\int_{\rr}\Big|\sum_{r=1}^{p}\theta_{r}\int_{0}^{t_r}
(s-y)_{+}^{\alpha-1}e^{-\lambda (s-y)_{+}}\ ds\Big|^{2}\ dy\\
&=\int_{\rr}\Big|\sum_{r=1}^{p}\theta_{r}\int_{0}^{t_r}
(s-y)_{+}^{H-\frac{3}{2}}e^{-\lambda (s-y)_{+}}\ ds\Big|^{2}\ dy
\end{split}
\end{equation*}
and this completes the proof.
\end{proof}

\begin{thm}\label{thm:fdd theorem}
The finite dimensional distribution of $n^{-H}S^{\frac{\lambda}{n}}(nt)$ converge in distribution to $Z^{1}_{H,\lambda}(t)$, given by \eqref{eq:THPdefn}, as $n\to\infty$.
That is
\begin{equation*}
\Big(n^{-H}S^{\frac{\lambda}{n}}(nt_1),n^{-H}S^{\frac{\lambda}{n}}(nt_2),\ldots,n^{-H}S^{\frac{\lambda}{n}}(nt_p)\Big)
\to \Big(Z^{1}_{H,\lambda}(t_1),Z^{1}_{H,\lambda}(t_2),\ldots,Z^{1}_{H,\lambda}(t_p)\Big)
\end{equation*}
as $n\to\infty$.
\end{thm}
\begin{proof}
Let $\theta_1,\theta_2,\ldots,\theta_p$, $t_1,t_2,\ldots,t_p\geq 0$. Then by computing the characteristic function of $n^{-H}\sum_{r=1}^{p}\theta_r S^{\frac{\lambda}{n}(nt_r)}$ we get
\begin{equation}\label{eq:converegence chf1}
\begin{split}
\mathbb{E}\Big[\exp\{in^{-H}\sum_{r=1}^{p}\theta_r S^{\frac{\lambda}{n}(nt_r)}\}\Big]
&=\mathbb{E}\Big[\exp\{in^{-H}\sum_{m\in\mathbb{Z}}\sum_{r=1}^{p}
\theta_r \xi^{\frac{\lambda}{n}}_{m}(nt_r)Z_m\}\Big]\\
&=\prod_{m\in\mathbb{Z}}\Big[\exp\{in^{-H}\sum_{r=1}^{p}
\theta_r \xi^{\frac{\lambda}{n}}_{m}(nt_r)Z_m\}\Big]\\
&=\prod_{m\in\mathbb{Z}}\exp\{{-n^{-2H}|\sum_{r=1}^{p}\theta_r \xi^{\frac{\lambda}{n}}_{m}(nt_r)|^{2}}\}\\
&=\exp\{{-\sum_{m\in\mathbb{Z}}n^{-2H}|\sum_{r=1}^{p}\theta_r \xi^{\frac{\lambda}{n}}_{m}(nt_r)|^2}\}.
\end{split}
\end{equation}
Taking the limit of \eqref{eq:converegence chf1} yields
\begin{equation*}
\begin{split}
&\lim_{n\to\infty}\mathbb{E}\Big[\exp\{i n^{-H}\sum_{r=1}^{p}\theta_r S^{\frac{\lambda}{n}(nt_r)}\}\Big]=
\exp\Big[-\lim_{n\to\infty}n^{-2H}\sum_{m\in\mathbb{Z}}\Big|\sum_{r=1}^{p}\theta_r \xi^{\frac{\lambda}{n}}_{m}(nt_r)\Big|^2\Big]\\
&=\exp\Big\{\int_{\rr}\Big|\sum_{r=1}^{p}\theta_{r}\int_{0}^{t_r}
(s-y)_{+}^{H-\frac{3}{2}}e^{-\lambda (s-y)_{+}}\ ds\Big|^{2}\ dy\Big\}\\
&=\mathbb{E}\Big[\exp\{{i\sum_{r=1}^{p}\theta_r Z^{1}_{H,\lambda}(t_r)}\}\Big]
\end{split}
\end{equation*}
as $n\to \infty$ and this completes the proof.
\end{proof}

\begin{thm}\label{thm:tightnessTheorem1}
Let $\{Z_j\}_{j\in\mathbb{Z}}$ be a sequence of i.i.d random variables with mean zero and finite variance. Then $n^{-H}S^{\frac{\lambda}{n}}(nt)$ converges weakly to $Z^{1}_{H,\lambda}(t)$, given by \eqref{eq:THPdefn}, in $C[0,1]$ ,as $n\to\infty$ $(C[0,1]$ is the space of all continuous functions defined on $[0,1])$. That is
\begin{equation}
n^{-H}S^{\frac{\lambda}{n}}(nt)\Rightarrow Z^{1}_{H,\lambda}(t),
\end{equation}
where $\Rightarrow$ means weak convergence in $C[0,1]$.
\end{thm}
\begin{proof}
In Theorem \ref{thm:fdd theorem}, We have shown the finite dimensional convergence of
$n^{-H}S^{\frac{\lambda}{n}}(nt)$ to $Z^{1}_{H,\lambda}(t)$. Therefore, here, we just need to prove the tightness of
$n^{-H}S^{\frac{\lambda}{n}}(nt)$. We show that for $0\leq t_1 \leq t_2\leq 1$,
\begin{equation}
\mathbb{E}\Big[\Big|n^{-H}S^{\frac{\lambda}{n}}(nt_2)-n^{-H}S^{\frac{\lambda}{n}}(nt_1)\Big|^{2}\Big]\leq C |t_2-t_1|^{2H},
\end{equation}
where $C$ is a constant. First apply \eqref{eq:Xi and coefficient connection} to get
\begin{equation}\label{eq:tightness1}
\begin{split}
&\sum_{m\in\mathbb{Z}}|\xi^{\frac{\lambda}{n}}_{m}(nt_2)-\xi^{\frac{\lambda}{n}}_{m}(nt_1)|\\
&\leq \Big(\frac{n}{\lambda}\Big)^{2\alpha}\int_{\rr}\Big|\int_{\frac{\lambda}{n}([nt_1]-x)}^{\frac{\lambda}{n}([nt_2]-x)}\omega^{\alpha-1}_{+}e^{-(\omega)_{+}}\ d\omega\Big|^{2}\ dx\\
&=\int_{\rr}\Big|\int_{[nt_1]-x}^{[nt_2]-x}y^{\alpha-1}_{+}e^{-\frac{\lambda}{n}(y)_{+}}\ dy\Big|^{2}\ dx\qquad (y:=\frac{n\omega}{\lambda})\\
&=\int_{\rr}\Big|\int_{0}^{[nt_2]-[nt_1]}\Big(z+(ns-x)\Big)^{\alpha-1}_{+}e^{-\frac{\lambda}{n}(z+(ns-x))_{+}}\ dz\Big|^{2}\ dx\qquad (z:=y-(ns-x))\\
&\leq C\int_{\rr}\Big|(nt_2-x)_{+}^{\alpha}-(nt_1-x)_{+}^{\alpha}\Big|^{2}.
\end{split}
\end{equation}
Maejima \cite{Maejima} proved that
\begin{equation}\label{eq:maejima upper bound}
\int_{\rr}\Big||nt_2-x|^{\alpha}-|nt_1-x|^{\alpha}\Big|^{2}\leq n^{1+2\alpha}(t_2-t_1)^{1+2\alpha}.
\end{equation}
Hence by applying \eqref{eq:tightness1} and \eqref{eq:maejima upper bound} we have
\begin{equation*}
\begin{split}
&\mathbb{E}\Big[\Big|n^{-H}\Big(S^{\frac{\lambda}{n}}(nt_2)-S^{\frac{\lambda}{n}}(nt_1)\Big)\Big|^{2}\Big]\\
&=n^{-2H}\mathbb{E}\Big(\sum_{m\in\mathbb{Z}}\Big|\Big(\xi^{\frac{\lambda}{n}}_{m}(nt_2)-
\xi^{\frac{\lambda}{n}}_{m}(nt_1)\Big)Z_m\Big|^{2}                       \Big)\\
&=n^{-2H}\sigma^2\sum_{m\in\mathbb{Z}}\Big|\xi^{\frac{\lambda}{n}}_{m}(nt_2)-\xi^{\frac{\lambda}{n}}_{m}(nt_1)\Big|^{2}\\
&\leq C\sigma^2 n^{-2H}n^{1+2\alpha}|t_2-t_1|^{1+2\alpha}\quad (H=\alpha+\frac{1}{2})\\
&\leq C|t_2-t_1|^{2H}.
\end{split}
\end{equation*}
Thus the tightness of $n^{-H}S^{\frac{\lambda}{n}}(nt)$ follows from Theorem 12.3 in \cite{Bill} and this completes the proof.
\end{proof}

\begin{thm}\label{thm:ARTFIMACONVERGENCE}
Let $\alpha>0$ and $X^{\lambda}_{k}$ be the ARTFIMA $(0,\alpha,\lambda,0)$. Suppose
\begin{equation*}
T^{\frac{\lambda}{n}}(t):=\sum_{k=1}^{[t]}X^{\frac{\lambda}{n}}_{k}+(t-[t])X^{\frac{\lambda}{n}}_{[t]+1},\quad t\geq 0.
\end{equation*}
Then,
\begin{equation*}
n^{-H}T^{\frac{\lambda}{n}}(nt)\Rightarrow Z^{1}_{H,\lambda}(t)
\end{equation*}
as $n\to\infty$ in $C[0,1]$.
\end{thm}
\begin{proof}
It follows from Stirling's approximation that
\begin{equation*}
\omega^{\frac{\lambda}{n}}_{j}=(-1)^{j}\binom{-\alpha}{j}e^{-\frac{\lambda}{n}}\sim \frac{\alpha}{\Gamma(1+\alpha)}j^{\alpha-1}e^{-\frac{\lambda}{n}} =C^{\frac{\lambda}{n}}_{j} \quad\text{as $j\to\infty$,}
\end{equation*}
where $C^{{\lambda}}_{j}$ is from \eqref{eq:Cdef}, see \cite[p.\ 24]{FCbook}. Hence for any $\epsilon>0$ there exists some positive integer $N$ such that
\begin{equation}\label{eq:bounds}
(1-\epsilon)C^{\frac{\lambda}{n}}_{j}<\omega^{\frac{\lambda}{n}}_{j}<(1+\epsilon)C^{\frac{\lambda}{n}}_{j}
\end{equation}
for all $j>N$. It follows that
\begin{equation}\label{eq:upperboundresult}
\begin{split}
\sum_{j=0}^{\infty}|\omega^{\frac{\lambda}{n}}_{j}|^{2}&\leq
\Big[\sum_{j=0}^{N}|\omega^{\frac{\lambda}{n}}_{j}|^{2}+(1+\epsilon)^{2}
\sum_{j=N+1}^{\infty}|C^{\frac{\lambda}{n}}_{j}|^{2}\Big]\\
&\leq\Big[\sum_{j=0}^{N}|\omega^{\frac{\lambda}{n}}_{j}|^{2}+(1+\epsilon)^{2}
\sum_{j=0}^{\infty}|C^{\frac{\lambda}{n}}_{j}|^{2}\Big]\\
\end{split}
\end{equation}
and consequently,
\begin{equation}\label{eq:chflower}
\begin{split}
\mathbb{E}\Big[\exp\{i\theta \Delta^{-\alpha,\frac{\lambda}{n}}_{1}Z_t\}\Big]&=
\exp\left\{-\theta^2\sigma^2\sum_{j=0}^{\infty}\Big|{\omega}^{\frac{\lambda}{n}}_{j}\Big|^{2}\right\}\\
&\geq C_{1}\exp\left\{-(1+\epsilon^2)\theta^2\sigma^2\sum_{j=0}^{\infty}\Big|{C}^{\frac{\lambda}{n}}_{j}\Big|^{2}\right\}\\
&= C_{1}\mathbb{E}\Big[\exp\{i(1+\epsilon)\theta X^{\frac{\lambda}{n}}_{t}\}\Big],
\end{split}
\end{equation}
where $C_1=\exp\left\{-\theta^2\sigma^2\sum_{j=0}^{N}\Big|{\omega}^{\frac{\lambda}{n}}_{j}\Big|^{2}\right\}$ is a finite positive constant.
Similarly,
\begin{equation}\label{eq:lowerboundresult}
\sum_{j=0}^{\infty}|\omega^{\frac{\lambda}{n}}_{j}|^{2}\geq
\Big[\sum_{j=0}^{N}|\omega^{\frac{\lambda}{n}}_{j}|^{2}+(1-\epsilon)^{2}
\sum_{j=N+1}^{\infty}|C^{\frac{\lambda}{n}}_{j}|^{2}\Big],\\
\end{equation}
so that
\begin{equation}\label{eq:chfupper}
\begin{split}
\mathbb{E}\Big[\exp\{i\theta \Delta^{-\alpha,\frac{\lambda}{n}}_{1}Z_t\}\Big]&=
\exp\left\{-\theta^2\sigma^2\sum_{j=0}^{\infty}\Big|{\omega}^{\frac{\lambda}{n}}_{j}\Big|^{2}\right\}\\
&\leq C_{2}\exp\left\{-(1-\epsilon^2)\theta^2\sigma^2\sum_{j=0}^{\infty}\Big|{C}^{\frac{\lambda}{n}}_{j}\Big|^{2}\right\}\\
&=C_{2}\mathbb{E}\Big[\exp\{i(1-\epsilon)\theta X^{\frac{\lambda}{n}}_{t}\}\Big],
\end{split}
\end{equation}
where $C_2=\exp\left\{-\theta^2\sigma^2\sum_{j=0}^{N}\Big|{\omega}^{\frac{\lambda}{n}}_{j}\Big|^{2}\right\}$ is a finite positive constant. From \eqref{eq:chflower} and \eqref{eq:chfupper} we have  :
\begin{equation}\label{eq:lowerupper chf}
C_{1}\mathbb{E}\Big[\exp\{i(1+\epsilon)\theta X^{\frac{\lambda}{n}}_{t}\}\Big]\leq
\mathbb{E}\Big[\exp\{i\theta \Delta^{-\alpha,\frac{\lambda}{n}}_{1}Z_t\}\Big]\leq
C_{2}\mathbb{E}\Big[\exp\{i(1-\epsilon)\theta X^{\frac{\lambda}{n}}_{t}\}\Big]
\end{equation}
for any $\varepsilon>0$.
The proof now follows from \eqref{eq:lowerupper chf} and Theorem \ref{thm:fdd theorem} and Theorem \eqref{thm:tightnessTheorem1} by letting $\epsilon\to 0$.
\end{proof}

Next, we answer the second question that we had in the introduction. Our approach follows that of Pipiras and Taqqu \cite{PipirasTaqqu2}.

For $m\in\mathbb{N}\cup\{\infty\}$, we define the approximation
\begin{equation*}
f^{+}_{n,m}=\sum_{j=0}^{m}f\Big(\frac{j}{n}\Big)1_{[\frac{j}{n},\frac{(j+1)}{n}]}, \qquad f^{-}_{n,m}=\sum_{j=-m}^{-1}f\Big(\frac{j}{n}\Big)1_{[\frac{j}{n},\frac{(j+1)}{n}]},
\end{equation*}

\begin{equation*}
f^{+}_{n}=f^{+}_{n,\infty}, \qquad f^{-}_{n}=f^{+}_{n,\infty},\qquad f_{m}=f^{+}_{n}+f^{-}_{n}.
\end{equation*}
The following theorem answers the third question that we had in the introduction.
\begin{thm}\label{thm:the third question}
Let $\alpha>0$ and $X^{\lambda}_{j}$ be the ARTFIMA $(0,\alpha,\lambda,0)$ times series. Suppose also that the following, condition $A$, is satisfied:
\begin{equation*}
Condition A:\quad f,f^{\pm}_{n}\in\mathcal{A}_{2}, \|f^{\pm}_{n}-f^{\pm}_{n,m}\|_{\mathcal{A}_{2}}\to 0\quad as\ m\to\infty, \|f-f_n\|_{\mathcal{A}_{2}}\to 0\quad as\ n\to\infty.
\end{equation*}
Then,
\begin{equation*}
n^{-H}\sum_{k=0}^{+\infty}f\Big(\frac{k}{n}\Big)X^{\frac{\lambda}{n}}_{k}\rightarrow \int_{\rr}f(u)Z^{1}_{H,\lambda}(du)
\end{equation*}
in distribution as $n\to\infty$.
\end{thm}
\begin{proof}
For the proof, we suppose
\begin{equation*}
W^{\frac{\lambda}{n}}_{n}=\frac{1}{n^{\frac{1}{2}+\alpha}}\sum_{j=-\infty}^{+\infty}f\Big(\frac{j}{n}\Big)X^{\frac{\lambda}{n}}_{j},\quad \ W=\int_{\rr}f(u)Z^{1}_{H,\lambda}(du).
\end{equation*}
The Wiener integral $W$ is well-defined, since $f\in\mathcal{A}_{2}$. To show that the series $W^{\frac{\lambda}{n}}_{n}$ is well-defined, apply the spectral representation of $X^{\frac{\lambda}{n}}_{k}$ by Lemma \ref{lem:spectralrepresntation ARTFIMA} and write
\begin{equation*}
\begin{split}
\frac{1}{n^{\frac{1}{2}+\alpha}}\sum_{j=0}^{m}f\Big(\frac{j}{n}\Big)X^{\frac{\lambda}{n}}_{j}&=
\frac{1}{n^{\frac{1}{2}+\alpha}}\int_{-\pi}^{\pi}\Big(\sum_{j=0}^{m}f\Big(\frac{j}{n}\Big)e^{ijx}\Big)dZ_{\frac{\lambda}{n}}(x)\\
&=\frac{1}{n^{\frac{1}{2}+\alpha}}\int_{\rr}\Big(\sum_{j=0}^{m}f\Big(\frac{j}{n}\Big)e^{\frac{ij\omega}{n}}\Big)1_{[-\pi n,\pi n]}(\omega)dZ_{\frac{\lambda}{n}}\Big(\frac{\omega}{n}\Big)\\
&=\frac{1}{n^{\alpha}-\frac{1}{2}}\int_{\rr}\Bigg(\sum_{j=0}^{m}f\Big(\frac{j}{n}\Big)\frac{e^{\frac{i(j+1)\omega}{n}}-e^{\frac{ij\omega}{n}}}{i\omega}
\Bigg)\frac{\frac{i\omega}{n}}{e^{\frac{i\omega}{n}}-1}1_{[-\pi n,\pi n]}(\omega)dZ_{\frac{\lambda}{n}}\Big(\frac{\omega}{n}\Big)\\
&=\frac{1}{n^{\alpha}-\frac{1}{2}}\int_{\rr}\widehat{f_{n,m}}(\omega)\frac{\frac{i\omega}{n}}{e^{\frac{i\omega}{n}}-1}1_{[-\pi n,\pi n]}(\omega)dZ_{\frac{\lambda}{n}}\Big(\frac{\omega}{n}\Big)
\end{split}
\end{equation*}
and hence we get
\begin{equation*}
\begin{split}
\mathbb{E}\Bigg|\frac{1}{n^{\frac{1}{2}+\alpha}}\sum_{j=0}^{m}f\Big(\frac{j}{n}\Big)X^{\frac{\lambda}{n}}_{j}\Bigg|^{2}
&=\int_{\rr}\Big|\widehat{f_{n,m}}(\omega)\Big|^{2}\Bigg| \frac{\frac{i\omega}{n}}{e^{\frac{i\omega}{n}}-1}\Bigg|^{2}1_{[-\pi n,\pi n]}(\omega)
\frac{1}{n^{2\alpha}}h_{\frac{\lambda}{n}}\Big(\frac{\omega}{n}\Big)\ d\omega\\
&\leq C\int_{\rr}\Big|\widehat{f_{n,m}}(\omega)\Big|^{2}\frac{1}{n^{2\alpha}}\Bigg(\frac{\lambda^2}{n^2}+\frac{\omega^2}{n^2}\Bigg)^{-\alpha}\ d\omega.\\
&=C\|f^{+}_{n,m}\|^{2}_{\mathcal{A}_{2}}
\end{split}
\end{equation*}
Then, by Condition $A$,
\begin{equation*}
\mathbb{E}\Bigg|\frac{1}{n^{\frac{1}{2}+\alpha}}\sum_{j=m_1+1}^{m_2}f\Big(\frac{j}{n}\Big)X^{\frac{\lambda}{n}}_{j}\Bigg|^{2}\leq C\|f^{+}_{n,m_2}-f^{+}_{n,m_1}\|^{2}_{\mathcal{A}_{2}}\to 0
\end{equation*}
as $m_1,m_2\to \infty$.

We next show that $W^{\frac{\lambda}{n}}_{n}$ convergence to $W$
in distribution. Recall from Theorem \ref{thm:A2space} that the set of elementary functions are dense in $\mathcal{A}_{2}$ and then there exists a sequence of elementary functions $f^{l}$ such that
$\|f^{l}-f\|_{\mathcal{A}_{2}}\to 0$, as $l\to\infty$. Now, assume
\begin{equation*}
W^{\frac{\lambda}{n},l}_{n}=\frac{1}{n^{\frac{1}{2}+\alpha}}\sum_{j=-\infty}^{+\infty}f^{l}\Big(\frac{j}{n}\Big)X^{\frac{\lambda}{n}}_{j},\quad \ W^{l}=\int_{\rr}f^{l}(u)Z^{1}_{H,\lambda}(du).
\end{equation*}
Observe that $W^{\frac{\lambda}{n},l}_{n}$ is well-define, since $W^{\frac{\lambda}{n},l}_{n}$ has a finite number elements and the elementary function $f^{l}$ is in $\mathcal{A}_{2}$. According to Theorem 4.2. in Billingsley \cite{Bill}, the series $W^{\frac{\lambda}{n}}_{n}$ convergence in distribution to $W$ if
\begin{description}
  \item[Step 1] $W^{l}\to W$, as $l\to\infty$,
  \item[Step 2] for all $l\in\mathbb{N}$, $W^{\frac{\lambda}{n},l}_{n}\to W^{l}$, as $n\to\infty$,
  \item[Step 3] $\limsup_{l}\limsup_{n}\mathbb{E}\Big|W^{\frac{\lambda}{n},l}_{n}-W^{\frac{\lambda}{n}}_{n}\Big|^{2}=0$.
\end{description}

Step $1$: The random variables $W^{l}$ and $W$ have normal distribution with mean zero and finite variance $\|f^{l}\|_{\mathcal{A}_{2}}$ and $\|f\|_{\mathcal{A}_{2}}$, respectively (See Theorem \ref{thm:A2space} and Definition \ref{defn:stochIntFourier}). Therefore $\mathbb{E}\Big|W^{l}-W\Big|^{2}=\|f^{l}-f\|_{\mathcal{A}_{2}}\to 0$ as $l\to\infty$.

Step $2$: Observe that $W^{\frac{\lambda}{n},l}_{n}=\int_{\rr}f^{l}(u)T^{\frac{\lambda}{n}}(du)$. Because $f^l$ is an elementary function, then the integral $W^{\frac{\lambda}{n},l}_{n}$
depends on the process $T^{\frac{\lambda}{n}}$ through a finite number of the points only. Now, Theorem \ref{thm:fdd theorem} and Theorem \ref{thm:ARTFIMACONVERGENCE} imply that $W^{\frac{\lambda}{n},l}_{n}\to W^{l}$, in distribution ,as $n\to\infty$, for all $i\in\mathbb{N}$.

Step $3$: For this step, we follow the same way as Pipiras and Taqqu did in Theorem 3.2 \cite{PipirasTaqqu2}. We have $\mathbb{E}\Big|W^{\frac{\lambda}{n},l}_{n}-W^{\frac{\lambda}{n}}_{n}\Big|^{2}\leq C\|f^{l}_{n}-f_{n}\|_{\mathcal{A}_{2}}$, where
\begin{equation*}
f^{l}_{n}:=\sum_{j}f^{l}\Big(\frac{j}{n}\Big)1_{(\frac{j}{n},\frac{(j+1)}{n})}(u).
\end{equation*}
Note that $f^{l}$ is an elementary function and therefore $\widehat{f^{l}_{n}}$ converges to $\widehat{f^{l}}$ at every point and $\Big|\widehat{f^{l}_{n}}(\omega)-\widehat{f^{l}}(\omega)\Big|\leq \widehat{g^{l}}(\omega)$ uniformly in $n$, for some function $\widehat{g^{l}}(\omega)$ which is bounded by $C_1$ and $C_2|\omega|^{-1}$ for all $\omega\in\rr$ (See Theorem 3.2. in \cite{PipirasTaqqu2} for more details). Let $\mu^{\alpha}(d\omega)=|\omega|^{-2\alpha}d\omega$ and $\mu^{\alpha}_{\lambda}(d\omega)=|\lambda^2+\omega^2|^{-2\alpha}d\omega$ be the measures on the real line for $\alpha>0$. Then apply the dominated converges theorem to see that
\begin{equation*}
\begin{split}
\|f^{l}_{n}-f^{l}\|^{2}_{\mathcal{A}_{2}}&=\|\widehat{f^{l}_{n}}-\widehat{f^{l}}\|_{L^{2}(\rr,\mu^{\alpha}_{\lambda})}\\
&\leq \|\widehat{f^{l}_{n}}-\widehat{f^{l}}\|_{L^{2}(\rr,\mu^{\alpha})}\to 0,
\end{split}
\end{equation*}
as $n\to\infty$. Hence by Condition $A$, the $\limsup_{n}\mathbb{E}\Big|W^{\frac{\lambda}{n},l}_{n}-W^{\frac{\lambda}{n}}_{n}\Big|^{2}\leq C\|f^{l}-f\|^{2}_{\mathcal{A}_{2}}\to 0$ as $l\to\infty$
and this completes the proof.
\end{proof}

\begin{rem}
{\emph{The result of Theorem \ref{thm:the third question} can also be derived by the following condition:
\begin{equation*}
Condition B:\quad f,f^{\pm}_{n}\in\mathcal{A}_{1}, \|f^{\pm}_{n}-f^{\pm}_{n,m}\|_{\mathcal{A}_{1}}\to 0\quad as\ m\to\infty, \|f-f_n\|_{\mathcal{A}_{1}}\to 0\quad as\ n\to\infty.
\end{equation*}
Since we have
\begin{equation*}\begin{split}
\ip fg_{{\mathcal{A}}_{1}}=\Gamma(H-\frac{1}{2})^2\ip {\varphi_f}{\varphi_g}_2&=\Gamma(1-\alpha)^2\ip {\hat\varphi_f}{\hat\varphi_g}_2\\
&=\Gamma(1-\alpha)^2\int_{-\infty}^{+\infty}\widehat{f}(\omega)\overline{\widehat{g}(\omega)}(\lambda^2+\omega^2)^{\frac{1}{2}-H}d\omega=\ip fg_{{\mathcal{A}}_{2}} \end{split}\end{equation*}
by the Plancherel Theorem.}}
\end{rem}

\section{Appendix}

Here we recall the definitions of tempered fractional integrals and derivatives and their properties that we used in the pervious sections.

\begin{defn}\label{defn:Tempered fractional integral}
For any $f\in{L}^{p}({\mathbb{R}})$ (where $1\leq p<\infty$), the positive and negative tempered fractional integrals are defined by
\begin{equation}\label{eq:positivetempered fractional integral}
{\mathbb I}^{\alpha,\lambda}_{+} f(t)=\frac{1}{\Gamma(\alpha)}\int_{-\infty}^{+\infty} f(u)(t-u)_{+}^{\alpha-1}e^{-\lambda(t-u)_{+}}du
\end{equation}
and
\begin{equation}\label{eq:negativetempered fractional integral}
{\mathbb I}^{\alpha,\lambda}_{-} f(t)=\frac{1}{\Gamma(\alpha)}\int_{-\infty}^\infty f(u)(u-t)_{+}^{\alpha-1}e^{-\lambda(u-t)_{+}}du
\end{equation}
respectively, for any $\alpha>0$ and $\lambda >0$,  where ${\Gamma(\alpha)=\int_{0}^{+\infty}e^{-x}x^{\alpha-1}dx}$ is the Euler gamma function, and $(x)_{+}=x I(x>0)$.
\end{defn}

When $\lambda=0$ these definitions reduce to the (positive and negative) Riemann-Liouville fractional integral \cite{FCbook,oldham,Samko}, which extends the usual operation of iterated integration to a fractional order.  When $\lambda=1$, the operator \eqref{eq:positivetempered fractional integral} is called the Bessel fractional integral \cite[Section 18.4]{Samko}.

We state the following lemma without the proof. We refer the reader to see Lemma 2.2 in \cite{MeerschaertsabzikarSPA}.
\begin{lem}\label{lem:TFI and Lp}
For any $\alpha>0$, $\lambda>0$, and $p\geq 1$, ${\mathbb I}^{\alpha,\lambda}_{\pm}$ is a bounded linear operator on $L^{p}(\mathbb{R})$ such that
\begin{equation}\label{TFIbound}
\|{\mathbb I}^{\alpha,\lambda}_{\pm} f\|_{p}\leq \lambda^{-\alpha}\|f\|_{p}
\end{equation}
for all $f\in L^{p}(\mathbb{R})$.
\end{lem}

Next we discuss the relationship between tempered fractional integrals and Fourier transforms.
Recall that the Fourier transform
\[{\mathcal{F}}[f](\omega)=\hat{f}(\omega)=\frac{1}{\sqrt{2\pi}}\int_{-\infty}^{+\infty}e^{i\omega x}f(x)dx\]
for functions $f\in L^{1}(\mathbb{R})\cap L^{2}(\mathbb{R})$ can be extended to an isometry (a linear onto map that preserves the inner product) on $L^{2}(\mathbb{R})$ such that
\begin{equation}\label{eq:FouriertransformL2}
\widehat{f}(\omega)=\lim_{n\to\infty}\frac{1}{\sqrt{2\pi}}\int_{-n}^{n}e^{-i\omega x}f(x)\ dx
\end{equation}
for any $f\in L^{2}(\mathbb{R})$, see for example \cite[Theorem 6.6.4]{Kierat}.

\begin{lem}\label{lem:FourierTFI}
For any $\alpha>0$ and $\lambda>0$ we have
\begin{equation}\label{TFI-FTeq}
{\mathcal{F}}[{\mathbb I}^{\alpha,\lambda}_{\pm} f](\omega)=\hat{f}{(\omega)}(\lambda\pm i\omega)^{-\alpha}
 \end{equation}
for all $f\in{L}^{1}({\mathbb{R}})$ and all $f\in{L}^{2}({\mathbb{R}})$.
\end{lem}

\begin{proof}
See Lemma 6.6 in \cite{MeerschaertsabzikarSPA}.
\end{proof}

Next we consider the inverse operator of the tempered fractional integral, which is called a tempered fractional derivative.  For our purposes, we only require derivatives of order $0<\alpha<1$, and this simplifies the presentation.

\begin{defn}\label{TFDdef}
The positive and negative tempered fractional derivatives of a function $f:\R\to\R$ are defined as
\begin{equation}\label{eq:temperedfractionalderivativepositive}
{\mathbb{D}}^{\alpha,\lambda}_{+}f(t)={\lambda}^{\alpha}f(t)+\frac{\alpha}{\Gamma(1-\alpha)}\int_{-\infty}^{t}\frac{f(t)-f(u)}{(t-u)^{\alpha+1}}
\,e^{-\lambda(t-u)}du.
\end{equation}
and
\begin{equation}\label{eq:temperedfractionalderivativenegative}
{\mathbb{D}}^{\alpha,\lambda}_{-}f(t)={\lambda}^{\alpha}f(t)+\frac{\alpha}{\Gamma(1-\alpha)}\int_{t}^{+\infty}\frac{f(t)-f(u)}
{(u-t)^{\alpha+1}}\,e^{-\lambda(u-t)}du
\end{equation}
respectively, for any $0<\alpha<1$ and any $\lambda>0$.
\end{defn}

If $\lambda=0$, the definitions \eqref{eq:temperedfractionalderivativepositive} and \eqref{eq:temperedfractionalderivativenegative} reduce to the positive and negative Marchaud fractional derivatives \cite[Section 5.4]{Samko}.

Note that tempered fractional derivatives cannot be defined pointwise for all functions $f\in L^{p}(\mathbb{R})$, since we need $|f(t)-f(u)|\to 0$ fast enough to counter the singularity of the denominator $(t-u)^{\alpha+1}$ as $u\to t$.
We can extend the definition of tempered fractional derivatives to a suitable class of functions in $L^2(\R)$.  For any $\alpha>0$ and $\lambda>0$ we may define the fractional Sobolev space
\begin{equation}\label{def:fracSovolev}
W^{\alpha,2}(\R):=\{f\in L^2(\R):\int_{\R} (\lambda^2+\omega^2)^{\alpha}|\hat f(\omega)|^2\,d\omega<\infty\} ,
\end{equation}
which is a Banach space with norm $\|f\|_{\alpha,\lambda}=\|(\lambda^2+\omega^2)^{\alpha/2}\hat f(\omega)\|_2$.  The space $W^{\alpha,2}(\R)$ is the same for any $\lambda>0$ (typically we take $\lambda=1$)  and all the norms $\|f\|_{\alpha,\lambda}$ are equivalent, since $1+\omega^2\leq\lambda^2+\omega^2\leq \lambda^2(1+\omega^2)$ for all $\lambda\geq 1$, and $\lambda^2+\omega^2\leq1+\omega^2\leq \lambda^{-2}(1+\omega^2)$ for all $0<\lambda<1$.

\begin{defn}\label{TFDdef2}
The positive (resp., negative) tempered fractional derivative ${\mathbb D}^{\alpha,\lambda}_{\pm}f(t)$ of a function $f\in W^{\alpha,2}(\R)$ is defined as the unique element of $L^2(\R)$  with Fourier transform $\widehat{f}{(\omega)}(\lambda\pm i\omega)^{\alpha}$ for any $\alpha>0$ and any $\lambda>0$.
\end{defn}

\begin{lem}\label{lem:inversoperator}
For any $\alpha>0$ and $\lambda>0$, we have
\begin{equation}\label{eq:DIfisf}
{\mathbb D}^{\alpha,\lambda}_{\pm}{\mathbb I}^{\alpha,\lambda}_{\pm}f(t)=f(t)
\end{equation}
for any function $f\in L^{2}(\mathbb{R})$, and
\begin{equation}\label{eq:IDfisf}
{\mathbb I}^{\alpha,\lambda}_{\pm}{\mathbb D}^{\alpha,\lambda}_{\pm}f(t)=f(t)
\end{equation}
for any $f\in W^{\alpha,2}(\R)$.
\end{lem}

\begin{proof}
Given  $f\in L^{2}(\mathbb{R})$, note that $g(t)={\mathbb I}^{\alpha,\lambda}_{\pm}f(t)$ satisfies $\hat g(k)=\hat{f}{(\omega)}(\lambda\pm i\omega)^{-\alpha}$ by Lemma \ref{lem:FourierTFI}, and then it follows easily that $g\in W^{\alpha,2}(\mathbb{R})$. Definition \ref{TFDdef2} implies that
\begin{equation}\label{eq:FouriertransformD}
{\mathcal F}[{\mathbb D}^{\alpha,\lambda}_{\pm}{\mathbb I}^{\alpha,\lambda}_{\pm}f](\omega)={\mathcal F}[{\mathbb D}^{\alpha,\lambda}_{\pm}g](\omega)
=\widehat{g}(\omega)(\lambda\pm i\omega)^{\alpha} =\hat f(\omega) ,
\end{equation}
and then \eqref{eq:DIfisf} follows using the uniqueness of the Fourier transform.  The proof of \eqref{eq:IDfisf} is similar.
\end{proof}

Here we collect some well known facts about the modified Bessel function of the second kind and we refer the reader to (Chapter 9, \cite{abramowitz}) for more details. The modified Bessel function $K_{\nu}(x)$ is regular function of $x$. It satisfies the following simple inequality
\begin{equation*}
K_{\nu}(x)>0\quad\text{for all $x>0$, for all $\nu\in\rr$}
\end{equation*}
and it has the following asymptotic expansion:
\begin{equation*}
K_{\nu}(x)\sim 2^{|\nu|-1}\Gamma(|\nu|)x^{-|\nu|}\qquad (\nu\neq 0)
\end{equation*}
as $x\to 0$.

\end{document}